\documentclass[a4paper,11pt]{article}
\usepackage[T1]{fontenc}
\usepackage{amsfonts}
\usepackage{textcomp}
\usepackage{amsmath, amsthm, amssymb, mathrsfs}
\usepackage{graphicx, url}

\title{On the Simulation of General Tempered Stable Ornstein-Uhlenbeck Processes}
\author{Michael Grabchak\footnote{Email address: mgrabcha@uncc.edu}\ \ 
{\it University of North Carolina Charlotte}}

\begin{document}
\newtheorem{prop}{Proposition}
\newtheorem{thrm}{Theorem}
\newtheorem{defn}{Definition}
\newtheorem{cor}{Corollary}
\newtheorem{lemma}{Lemma}
\newtheorem{remark}{Remark}
\newtheorem{exam}{Example}

\newcommand{\rd}{\mathrm d}
\newcommand{\rE}{\mathrm E}
\newcommand{\ts}{\mathrm{TS}^p_{\alpha,d}}
\newcommand{\ID}{\mathrm{ID}}
\newcommand{\tr}{\mathrm{tr}}
\newcommand{\iid}{\stackrel{\mathrm{iid}}{\sim}}
\newcommand{\eqd}{\stackrel{d}{=}}
\newcommand{\cond}{\stackrel{d}{\rightarrow}}
\newcommand{\conv}{\stackrel{v}{\rightarrow}}
\newcommand{\conw}{\stackrel{w}{\rightarrow}}
\newcommand{\conp}{\stackrel{p}{\rightarrow}}
\newcommand{\simp}{\stackrel{p}{\sim}}

\maketitle

\begin{abstract}
We give an explicit representation for the transition law of a tempered stable Ornstein-Uhlenbeck process and use it to develop a rejection sampling algorithm for exact simulation of increments from this process. Our results apply to general classes of both univariate and multivariate tempered stable distributions and contain a number of previously studied results as special cases.\\

\noindent\textbf{Keywords:} Tempered stable distributions, Ornstein-Uhlenbeck processes, rejection sampling, selfdecomposability

\end{abstract}

\section{Introduction}

Non-Gaussian processes of Ornstein-Uhlenbeck-type (henceforth OU-processes) form a rich and flexible class of stochastic models. They are the continuous time analogues of autoregressive AR(1) processes and are important in the study of selfdecomposable distributions. Further, they are mean reverting, which makes them useful for many applications. We are particularly motivated by their applications to mathematical finance, where they have been used to model stochastic volatility \cite{Barndorff-Nielsen:Shephard:2001} \cite{Barndorff-Nielsen:Shephard:2002}, stochastic interest rates \cite{Norberg:2004} \cite{Janczura:etal:2011}, and commodity prices \cite{Benth:Kallsen:Meyer-Brandis:2007} \cite{Goncu:Akyildirim:2016}. Here we focus on the class of TSOU-processes, which are OU-processes with tempered stable limiting distributions. These distributions are obtained by modifying the tails of infinite variance stable distributions to make them lighter, which leads to models that are more realistic for a variety of application areas. In particular, discussions of financial applications can be found in \cite{Grabchak:Samorodnitsky:2010}, \cite{Rachev:Kim:Bianchi:Fabozzi:2011}, \cite{Grabchak:2016book}, and the references therein. Some of the earliest tempered stable distributions were introduced in the seminal paper Tweedie (1984) \cite{Tweedie:1984} and are called Tweedie distributions. A general framework was developed in Rosi\'nski (2007) \cite{Rosinski:2007}. This was further generalized, in several directions, in \cite{Rosinski:Sinclair:2010}, \cite{Bianchi:Rachev:Kim:Fabozzi:2011}, and \cite{Grabchak:2012}. A survey, along with a historical overview and many references can be found in \cite{Grabchak:2016book}. 

In this paper, we give an explicit representation for the transition law of a TSOU-process and use it to develop a simulation method  based on rejection sampling. This method extends the rejection sampling technique for simulating tempered stable L\'evy processes, which was introduced in \cite{Grabchak:2019}. 
Similar representations for the transition laws of TSOU-processes with Tweedie (and closely related) limiting distributions are given in \cite{Zhang:Zhang:2009}, \cite{Zhang:2011}, \cite{Kawai:Masuda:2011},  \cite{Kawai:Masuda:2012}, and \cite{Bianchi:Rachev:Fabozzi:2017}. Most of these are special cases of our result.  

The rest of this paper is organized as follows. In Section \ref{sec: aux} we discuss two distributions that will be useful for simulating TSOU-processes and in Section \ref{sec: gen TS} we recall the definition of the class of tempered stable distributions. Then, in Section \ref{sec: gen TSOU} we formally introduce TSOU-processes, give explicit representations of their transition laws, and introduce our methodology for simulation. In Section \ref{sec: pTaS} we specialize our results to the important class of $p$-tempered $\alpha$-stable distributions. A small simulation study is given in Section \ref{sec: sims}.  Proofs are postponed to Section \ref{sec: proofs}.

Before proceeding, we introduce some notation. Let $\mathbb R^d$ be the space of $d$-dimensional column vectors of real numbers equipped with the usual inner product $\langle\cdot,\cdot\rangle$ and the usual norm $|\cdot|$. Let $\mathbb S^{d-1}=\{x\in\mathbb R^d: |x|=1\}$ denote the unit sphere in $\mathbb R^d$. Let $\mathfrak B(\mathbb R^d)$ and $\mathfrak B(\mathbb S^{d-1})$ denote the Borel sets in $\mathbb R^d$ and $\mathbb S^{d-1}$, respectively. If $a,b\in\mathbb R$, we write $a\vee b$ and $a\wedge b$ to denote, respectively, the maximum and the minimum of $a$ and $b$. We write $U(0,1)$ to denote the uniform distribution on $[0,1]$, we write $\delta_x$ to denote a point mass at $x$, and we write $1_A$ to denote the indicator function on set $A$.  If $\mu$ is a probability measure on $\mathbb R^d$, we write $X\sim\mu$ to denote that $X$ is an $\mathbb R^d$-valued random variable with distribution $\mu$. We write $\eqd$ to denote equality in distribution and $\conw$ to denote weak convergence. 

\section{Auxiliary Distributions}\label{sec: aux}

In this section we briefly introduce two distributions, which will be important for simulating increments from TSOU-processes. These are the modified log-Laplace distribution and the generalized gamma distribution.

We begin with the modified log-Laplace distribution. Consider a distribution with probability density function (pdf) given by
\begin{eqnarray*}
g(x;\alpha, p,\delta) &=& \alpha(p-\alpha)\frac{\delta^\alpha}{\alpha\delta^p+p-\alpha}\left(x^{p-\alpha-1}1_{[0<x\le \delta]} + x^{-1-\alpha}1_{[x>\delta]}\right)\\
&=& \frac{\alpha \delta^p}{\alpha\delta^p+p-\alpha}\delta^{-(p-\alpha)}(p-\alpha)x^{p-\alpha-1}1_{[0<x\le\delta]} \\
&&\qquad\quad + \frac{p-\alpha}{\alpha\delta^p+p-\alpha} \alpha \delta^\alpha x^{-1-\alpha}1_{[x>\delta]},
\end{eqnarray*}
where $p>\alpha>0$ and $\delta>0$ are parameters. We call this the modified log-Laplace distribution and denote it by $\mathrm{MLL}(\alpha, p,\delta)$. This distribution is a mixture of the beta distribution with pdf
$$
(p-\alpha)\delta^{-(p-\alpha)}x^{p-\alpha-1}, \ 0<x\le\delta
$$
and the Pareto distribution with pdf
$$
\alpha \delta^\alpha x^{-1-\alpha},\ x>\delta.
$$
The mixing parameter is 
$$
h =  \frac{\alpha\delta^p}{\alpha\delta^p+p-\alpha}.
$$
Using this mixture representation, we can simulate from the distribution as follows.  If $U_1$ and $U_2$ are independent and identically distributed (iid) random variables from $U(0,1)$, then 
$$
Y = U_1^{1/(p-\alpha)}\delta1_{[U_2<h]} + U_1^{-1/\alpha}\delta1_{[U_2\ge h]} \sim \mathrm{MLL}(\alpha, p,\delta).
$$
The reason that we call this the modified log-Laplace distribution is because the log-Laplace distribution is also a mixture of a  beta distribution and a Pareto distribution, but with a different mixing parameter. The modified log-Laplace distribution is a log-Laplace distribution only when $\delta=1$. For more on the log-Laplace distribution see \cite{Kozubowski:Podgorski:2003}. 

We now turn to the generalized gamma distribution. Recall that the pdf of a gamma distribution is given by
$$
\frac{\theta^\beta}{\Gamma(\beta)}u^{\beta-1}e^{-\theta u}, \ \ u>0,
$$
where $\beta,\theta>0$ are parameters. We denote this distribution by $\mathrm{Ga}(\beta,\theta)$. The generalized gamma distribution was introduced in \cite{Stacy:1962}. It has a pdf of the form
$$
\frac{p\theta^{\beta/p}}{\Gamma(\beta/p)}u^{\beta-1}e^{-\theta u^p}, \ \  u>0,
$$
where $\beta,p,\theta>0$ are parameters. We denote this distribution by $\mathrm{GGa}(\beta,p,\theta)$. It is readily checked that if $X\sim \mathrm{Ga}(\beta/p,\theta)$, then $X^{1/p}\sim \mathrm{GGa}(\beta,p,\theta)$. Thus, to simulate from a generalized gamma distribution it suffices to know how to simulate from a gamma distribution. Approaches for doing this are well-known, see e.g.\ \cite{Ahrens:Dieter:1982:gamma} and the references therein. 

\section{Tempered Stable Distributions}\label{sec: gen TS}

An infinitely divisible distribution $\mu$ on $\mathbb R^d$ is a probability measure with a characteristic function of the form $\hat\mu(z) = \exp\{C_\mu(z)\}$, where, for $z\in\mathbb R^d$,
\begin{eqnarray*}
C_\mu(z) = -\frac{1}{2}\langle z,Az\rangle + i\langle b,z\rangle + \int_{\mathbb R^d}\left(e^{i\langle z,x\rangle}-1-i\langle z,x\rangle h(x)\right)M(\rd x).
\end{eqnarray*}
Here, $A$ is a symmetric nonnegative-definite $d\times d$-dimensional matrix called the Gaussian part, $b\in\mathbb R^d$ is called the shift, and $M$ is a Borel measure, called the L\'evy measure, which satisfies
\begin{eqnarray*}
M(\{0\})=0 \mbox{\ and\ } \int_{\mathbb R^d}(|x|^2\wedge1)M(\rd x)<\infty.
\end{eqnarray*}
The function $h:\mathbb R^d\mapsto\mathbb R$, which we call the $h$-function, can be any Borel function satisfying
\begin{eqnarray*}
\int_{\mathbb R^d}\left|e^{i\langle z,x\rangle}-1-i\langle z,x\rangle h(x)\right| M(\rd x)<\infty
\end{eqnarray*}
for all $z\in\mathbb R^d$. For a fixed $h$-function, $A$, $M$, and $b$ uniquely determine the distribution $\mu$, and we write $\mu=\ID_d(A,M,b)_h$. The choice of $h$ does not affect $A$ and $M$, but different choices of $h$ require different values for $b$, see Section 8 in \cite{Sato:1999}.

Associated with every infinitely divisible distribution $\mu=\ID_d(A,M,b)_h$ is a stochastic process  $\{X_t:t\ge0\}$, which is called a L\'evy process. This processes is stochastically continuous and has independent and stationary increments. The characteristic function of $X_t$ is $\left(\hat\mu(z)\right)^t$. It follows that, for each $t\ge0$, $X_t\sim \ID_d(tA,tM,tb)_h$. For more on infinitely divisible distributions and their associated L\'evy processes see \cite{Sato:1999}.

Selfdecomposable distributions are an important class of infinitely divisible distributions. They correspond to the case, where the L\'evy measure $M$ is of the form
\begin{eqnarray*}
M(B) = \int_{\mathbb S^{d-1}}\int_0^\infty 1_{B}(u\xi) k(\xi,u)u^{-1} \rd u \sigma(\rd \xi), \ \ B\in\mathfrak B(\mathbb R^d).
\end{eqnarray*}
Here $\sigma$ is a finite Borel measure on $\mathbb S^{d-1}$ and $k:\mathbb S^{d-1}\times(0,\infty) \mapsto[0,\infty)$ is a Borel function, which, for each fixed $\xi\in\mathbb S^{d-1}$, is nonincreasing in $u$. A more intuitive characterization is as follows. A distribution $\mu$ is selfdecomposable if and only if there is an $X\sim \mu$ such that, for all $c>1$, there exists a random variable $Y_c$, independent of $X$, satisfying 
\begin{eqnarray*}
X\eqd \frac{1}{c}X+Y_c.
\end{eqnarray*}

Following \cite{Rosinski:Sinclair:2010}, we define a tempered stable distribution on $\mathbb R^d$ as an infinitely divisible distribution with no Gaussian part and a L\'evy measure of the form
\begin{eqnarray}\label{eq: levy ts}
L_\alpha(B) = \int_{\mathbb S^{d-1}}\int_0^\infty 1_{B}(u\xi) u^{-1-\alpha} q(\xi,u)\rd u \sigma(\rd \xi), \ \ B\in\mathfrak B(\mathbb R^d),
\end{eqnarray}
where $\alpha\in(0,2)$, $q:\mathbb S^{d-1}\times(0,\infty) \mapsto[0,\infty)$ is a Borel function, and $\sigma$ is a finite Borel measure on $\mathbb S^{d-1}$. The function $q$ is called the tempering function. Throughout this paper we assume that $q$ satisfies the following two assumptions:
\begin{itemize}
\item[\textbf{A1.}] $0\le q(\xi,u) \le 1$ for all $u>0$ and $\xi\in\mathbb S^{d-1}$, and
\item[\textbf{A2.}] $q(\xi,u)$ is nonincreasing in $u$ for each fixed $\xi\in\mathbb S^{d-1}$.
\end{itemize} 
Assumption A1 ensures that the tempering function is bounded, while Assumption A2 ensures that the corresponding distribution is selfdecomposable. For tempered stable distributions we use the $h$-function given by
\begin{eqnarray*}
h_\alpha(x) = \left\{\begin{array}{ll}
0& \alpha\in(0,1)\\
1_{[|x|\le1]} & \alpha=1\\
1 & \alpha\in(1,2)
\end{array}\right.,
\end{eqnarray*}
and we denote the distribution $\ID_d(0,L_\alpha,b)_{h_\alpha}$ by $\mathrm{TS}_{\alpha,d}(\sigma,q,b)$. When $q(\xi,u) \equiv 1$, the distribution $\mathrm{TS}_{\alpha,d}(\sigma,q,b)$ is $\alpha$-stable and we sometimes write $\mathrm{S}_{\alpha,d}(\sigma,b)$ in this case. The characteristic function of $\mathrm{TS}_{\alpha,d}(\sigma,q,b)$ is given by
\begin{eqnarray*}
\exp\left\{i\langle b,z\rangle + \int_{\mathbb S^{d-1}}\int_0^\infty \psi_\alpha(z,u\xi) q(\xi,u)u^{-1-\alpha}\rd u\sigma(\rd\xi)\right\}, \ \ z\in\mathbb R^d,
\end{eqnarray*}
where
\begin{eqnarray*}
\psi_\alpha(z,x) = e^{i\langle z,x \rangle} - 1 - i\langle z,x \rangle h_\alpha(x) \mbox{ for } z,x\in\mathbb R^d.
\end{eqnarray*}

\begin{remark}
The name ``tempering function'' can be explained as follows. It is often assumed that $q$ satisfies the additional assumption that for each $\xi\in\mathbb S^{d-1}$
 \begin{eqnarray}\label{eq: additional assumption}
 \lim_{u\to0}q(\xi,u)=1 \mbox{ and } \lim_{u\to\infty}q(\xi,u)=0.
 \end{eqnarray}
In this case, the tempered stable distribution $\mathrm{TS}_{\alpha,d}(\sigma,q,b)$ has a distribution function that looks like that of the stable distribution $\mathrm{S}_{\alpha,d}(\sigma,b)$ in some central region, but with lighter tails. In this sense, $\mathrm{TS}_{\alpha,d}(\sigma,q,b)$  ``tempers'' the tails of $\mathrm{S}_{\alpha,d}(\sigma,b)$.  While this is part of the motivation for defining tempered stable distributions, we will not require \eqref{eq: additional assumption} to hold in this paper. 
\end{remark}

\section{TSOU-Processes}\label{sec: gen TSOU}

We begin this section by recalling the definition of a process of Ornstein-Uhlenbeck-type (OU-process). Toward this end, let $Z=\{Z_t:t\ge0\}$ be a L\'evy process with $Z_1\sim \ID_d(A,M,b)_h$ and define the process $\{Y_t:t\ge0\}$ by the stochastic differential equation
$$
\rd Y_t = -\lambda Y_t\rd t + \rd Z_t,
$$
where $\lambda>0$ is a parameter. This has a strong solution of the form
\begin{eqnarray*}
Y_t = e^{-\lambda t}Y_0 + \int_0^t e^{-\lambda(t-s)} \rd Z_s.
\end{eqnarray*}
In this case $Y=\{Y_t:t\ge0\}$ is called an OU-process with parameter $\lambda$ and $Z$ is called the background driving L\'evy process (BDLP). The process $Y$ is a Markov process and, so long as
\begin{eqnarray*}
\int_{|x|>2} \log|x| M(\rd x)<\infty,
\end{eqnarray*}
it has a limiting distribution. This distribution is necessarily selfdecomposable. Further, every selfdecomposable distribution is the limiting distribution of some OU-process. For details see \cite{Sato:1999} or \cite{Rocha-Arteaga:Sato:2003}. 

Since $\mathrm{TS}_{\alpha,d}(\sigma,q,b)$ is selfdecomposable, it is the limiting distribution of some OU-process, which we call a TSOU-process. We now characterize the BDLP and the transition law of this process.

\begin{thrm}\label{thrm: trans char func}
Consider the distribution $\mathrm{TS}_{\alpha,d}(\sigma,q,b)$ and denote its L\'evy measure by $L_\alpha$. Let $Y=\{Y_t:t\ge0\}$ be an OU-process with parameter $\lambda>0$ and BDLP $Z=\{Z_t:t\ge0\}$. Assume that $Z_1\sim \ID_d(0, \lambda M, \lambda c_\alpha)_{h_\alpha}$ with
$$
c_\alpha = b - \left\{\begin{array}{ll}
\int_{|x|>1} \frac{x}{|x|} M(\rd x) & \mbox{if  }\alpha=1\\
0 & \mbox{otherwise}
\end{array}\right.,
$$
and
$$
M(B) = \int_{\mathbb S^{d-1}} \int_0^\infty 1_B(u\xi) \rd \rho_\xi(u) \sigma(\rd \xi), \ \ B\in\mathfrak B(\mathbb R^d),
$$
where $\rho_\xi(u) = -u^{-\alpha} q(\xi,u)$. In this case, $Y$ is a Markov process with temporally homogenous transition function $P_t(y,\rd x)$ having characteristic function
$\int_{\mathbb R^d}e^{i\langle x,z\rangle} P_t(y,\rd x) = \exp\left\{C_t(y,z)\right\}$, where
\begin{eqnarray*}
C_t(y,z) &=&  ie^{-\lambda t}\langle y,z\rangle +  i  \langle  d_{\alpha},z\rangle + \left(1-e^{-\lambda t\alpha}\right)\int_{\mathbb R^d}\psi_\alpha(z,x) L_\alpha(\rd x)  \\
&&\ + e^{-\alpha\lambda t} \int_{\mathbb S^{d-1}}\int_0^\infty \psi_\alpha(z,u\xi) \left( q(\xi,u) - q(\xi,ue^{\lambda t}) \right)  u^{-1-\alpha} \rd u  \sigma(\rd \xi),\\
d_{\alpha} &=&   \left(1-e^{-\lambda t}\right) b +  \left\{\begin{array}{ll} 
\tilde d_1 &  \mbox{if  } \alpha=1\\
0 & \mbox{otherwise}
\end{array}\right.,
\end{eqnarray*}
and
$$
\tilde d_1 =\int_{1<|x|\le e^{\lambda t}}  x \left(|x|^{-1}-e^{-\lambda t}\right) M(\rd x) - \left(1-e^{-\lambda t}\right)\int_{|x|>1} \frac{x}{|x|} M(\rd x).
$$
Further, for any $y\in\mathbb R^d$, $P_t(y,\cdot)\conw \mathrm{TS}_{\alpha,d}(\sigma,q,b)$ as $t\to\infty$.
\end{thrm}

In the above $\rho_\xi$ is a nondecreasing function and the integral with respect to $\rd \rho_\xi(u)$ is the Stieltjes integral. We now give an explicit representation for an increment of a TSOU-process, under an additional assumption.

\begin{thrm}\label{thrm: main gen}
Let $\{Y_t:t\ge0\}$ be a TSOU-process with parameter $\lambda>0$ and limiting distribution $\mathrm{TS}_{\alpha,d}(\sigma,q,b)$. Fix $t>0$ and let
\begin{eqnarray}\label{eq:K}
K = \int_{\mathbb S^{d-1}}\int_0^\infty  \left( q(\xi,u) - q(\xi,ue^{\lambda t}) \right)  u^{-1-\alpha} \rd u \sigma(\rd \xi) .
\end{eqnarray}
If $K=0$, then, given $Y_s=y$, we have
\begin{eqnarray}\label{eq: trans RV stable}
Y_{s+t} \eqd e^{-\lambda t} y + d_{\alpha} + X_0
\end{eqnarray}
and if $0<K<\infty$, then, given $Y_s=y$, we have
\begin{eqnarray}\label{eq: trans RV}
Y_{s+t} \eqd e^{-\lambda t} y + d_{\alpha} + X_0 + \sum_{j=1}^{N} X_j - \psi.
\end{eqnarray}
In the above
$$
\psi =  e^{-\alpha\lambda t} \int_{\mathbb S^{d-1}}\int_0^\infty h_\alpha(u\xi) \left( q(\xi,u) - q(\xi,ue^{\lambda t}) \right)  u^{-\alpha} \rd u \xi \sigma(\rd \xi)
$$
and $N, X_0, X_1, X_2, X_3, \dots$ are independent random variables with:\\
1. $X_0\sim \mathrm{TS}_{\alpha,d}(\sigma_0,q,0)$ where $\sigma_0(\rd \xi) = (1-e^{-\alpha \lambda t})\sigma(\rd \xi)$,\\
2. $X_1, X_2, X_3, \dots$ are iid random variables such that $X_i\eqd \xi W$, where $\xi\in\mathbb S^{d-1}$, $W>0$, and the joint distribution of $\xi$ and $W$ is
$$
H(\rd \xi,\rd u)=\frac{1}{K} \left( q(\xi,u) - q(\xi,ue^{\lambda t}) \right)  u^{-1-\alpha} \rd u  \sigma(\rd \xi), \ \ u>0,\ \xi\in\mathbb S^{d-1},
$$
3. $N$ has a Poisson distribution with mean $Ke^{-\alpha \lambda t}$.
\end{thrm}

\begin{remark}
1. If $q(\xi,u)\equiv1$, then $\mathrm{TS}_{\alpha,d}(\sigma,q,b) =\mathrm S_{\alpha,d}(\sigma,b)$. In this case $K=0$ and hence the transition law for an OU-process with an $\alpha$-stable limiting distribution can be represented as in \eqref{eq: trans RV stable}. 2. If $\alpha\in(0,1)$ then $h_\alpha\equiv0$ and hence $\psi=0$.
\end{remark}

Theorem \ref{thrm: main gen} provides a simple recipe for simulating an increment from a TSOU-process when $K<\infty$. The main ingredients are the ability to simulate from three distributions: the Poisson distribution, the $\mathrm{TS}_{\alpha,d}(\sigma_0,q,0)$ distribution, and distribution $H$. The problem of simulating from a Poisson distribution is well-studied, see e.g.\ \cite{Ahrens:Dieter:1982}. Under mild conditions, a rejection sampling technique for simulating from $\mathrm{TS}_{\alpha,d}(\sigma_0,q,0)$  is given in \cite{Grabchak:2019}. To simulate from $H$ we first define the quantities
$$
\kappa_\xi = \frac{1}{\int_0^\infty \left( q(\xi,u) - q(\xi,ue^{\lambda t}) \right) u^{-1-\alpha}\rd u}, \ \ \xi\in\mathbb S^{d-1},
$$
the probability measure on $\mathbb S^{d-1}$ given by
$$
\sigma_1(\rd \xi) = \frac{1}{\kappa_\xi K}\sigma(\rd \xi),
$$
and the family of pdfs on $(0,\infty)$ given by
\begin{eqnarray}\label{eq: f xi}
f_\xi(u)= \kappa_\xi \left( q(\xi,u) - q(\xi,ue^{\lambda t}) \right)  u^{-1-\alpha}  , \ \ u>0.
\end{eqnarray}
It is readily checked that
$$
H(\rd \xi,\rd u)= f_\xi(u) \rd u \sigma_1(\rd \xi), \ \ u>0,\ \xi\in\mathbb S^{d-1}.
$$ 
Hence, one can simulate $(\xi,W)$ from $H$ by first simulating $\xi$ from $\sigma_1$ and then $W$ from the distribution with pdf $f_\xi$. Thus, the problem reduces to that of simulating from $\sigma_1$ and from the pdf $f_\xi$.

We begin by discussing simulation from $\sigma_1$. One of the most important situations is when the support of $\sigma_1$ is finite. This always holds, in particular, when the dimension $d=1$. In this case, the problem reduces to simulating from a multinomial distribution. Another important situation is when $\sigma_1$ follows a uniform distribution. This problem is well-studied, see e.g.\ \cite{Tashiro:1977}. For a discussion of simulation from a variety of other standard distributions on $\mathbb S^{d-1}$, see the monograph \cite{Johnson:1987}. While no method works in general, one can often set up an approximate simulation method by first approximating $\sigma_1$ by a distribution with a finite support, see Lemma 1 in \cite{Byczkowski:Nolan:Rajput:1993}.

We now discuss simulation from $f_\xi$ in two important situations with $K<\infty$. The first, which we call ``hard truncation,'' covers a useful but fairly specific situation, while the second, which we call ``Class F,'' is very general.

\subsection{Hard Truncation}

Consider the distribution $\mathrm{TS}_{\alpha,d}(\sigma,q,b)$, where $q(\xi,u) = 1_{[0\le u\le\gamma_\xi]}$. Assume that $\gamma_\xi\ge0$ for each $\xi\in\mathbb S^{d-1}$, that the function $\xi\mapsto \gamma_\xi$ is measurable, and that 
$
\int_{\mathbb S^{d-1}} \gamma_\xi^{-\alpha} \sigma(\rd \xi)<\infty.
$ 
Such distributions appear in certain limit theorems, see \cite{Chakrabarty:Samorodnitsky:2012} and \cite{Grabchak:2018}. In this case, it is readily checked that $K<\infty$ and that the pdf $f_\xi$, as given in \eqref{eq: f xi}, can be written as
$$
f_\xi(u)=  \frac{\alpha\gamma_\xi^\alpha}{(e^{\alpha\lambda t}-1)} u^{-1-\alpha}  , \ \ e^{-\lambda t}\gamma_\xi<u\le \gamma_\xi.
$$
The corresponding cumulative distribution function (cdf) can be written as
$$
F_\xi(x)=  \frac{\gamma_\xi^\alpha}{(e^{\alpha\lambda t}-1)} \left(e^{\alpha\lambda t}\gamma_\xi^{-\alpha}-x^{-\alpha}\right), \ \ x\in(e^{-\lambda t}\gamma_\xi, \gamma_\xi)
$$
and, for $y\in(0,1)$,
$$
F_\xi^{-1}(y) = \gamma_\xi\left(e^{\alpha\lambda t}-\left(e^{\alpha\lambda t}-1\right)y\right)^{-1/\alpha}.
$$
Hence, we can use the inverse transform method to simulate $W$ from $f_\xi$ by first simulating $U\sim U(0,1)$ and then taking $W=F_\xi^{-1}(U)$.

\subsection{Class F}
 
 Consider the distribution $\mathrm{TS}_{\alpha,d}(\sigma,q,b)$ and assume that there exists an $\epsilon>0$ and Borel functions $p(\cdot),M_\bullet:\mathbb S^{d-1}\mapsto [0,\infty)$ with
\begin{eqnarray*}
\inf_{\xi\in\mathbb S^{d-1}}p(\xi)>\alpha,  \ \ \ \int_{\mathbb S^{d-1}} M_\xi \left(e^{p(\xi)\lambda t}-1\right)\sigma(\rd \xi)<\infty,
\end{eqnarray*}
such that for any $u,v\in(0,\epsilon)$ and $\sigma$-a.e.\ $\xi$
\begin{eqnarray*}
\left| q(\xi,u) - q(\xi,v)\right| \le M_\xi \left| u^{p(\xi)} - v^{p(\xi)}\right|.
\end{eqnarray*}
In this case, we say that $q$ belongs to Class F and write $q\in\mathrm{CF}_{\alpha,d}(\epsilon,M_\bullet,p(\cdot),\sigma)$. For simplicity, we will sometimes also say that the distribution $\mu=\mathrm{TS}_{\alpha,d}(\sigma,q,b)$ belongs to Class F and write $\mu\in\mathrm{CF}_{\alpha,d}(\epsilon,M_\bullet,p(\cdot),\sigma)$. It is easily verified that, in this case, $K<\infty$.

It is not always easy to check if a given tempering function belongs to Class F. A sufficient condition is that $q(\xi,u) = q_1(\xi,u^{p(\xi)})$, where, for $\sigma$-a.e.\ $\xi$, the function $q_1(\xi,\cdot)$ is uniformly Lipschitz continuous on a neighborhood of zero. By this, we mean that there are constants $M',\epsilon'>0$ such that for  any $u,v\in(0,\epsilon')$ and $\sigma$-a.e.\ $\xi\in\mathbb S^{d-1}$
\begin{eqnarray}\label{eq: q1}
\left| q_1(\xi,u)-q_1(\xi,v)\right|\le M'|u-v|.
\end{eqnarray}
In this case $q\in\mathrm{CF}_{\alpha,d}(\epsilon,M_\bullet,p(\cdot),\sigma)$ with $\epsilon=\epsilon'$ and $M_\bullet\equiv M'$. In particular, by the mean value theorem,  \eqref{eq: q1} holds whenever $\left|\frac{\partial}{\partial u}q_1(\xi,u)\right|\le M'$ for all $u\in(0,\epsilon')$ and $\sigma$-a.e.\ $\xi\in\mathbb S^{d-1}$.

We now develop a rejection sampling approach for simulating from $f_\xi$ in this case.  The approach is based on the following result.

\begin{prop}\label{prop:for AR}
If $q\in\mathrm{CF}_{\alpha,d}(\epsilon,M_\bullet,p(\cdot),\sigma)$ and $f_\xi$ is as in \eqref{eq: f xi}, then
$$
f_\xi(u)\le V_1 g_1(u), \ \ u>0,
$$
where $g_1$ is the pdf of $\mathrm{MLL}(\alpha,p(\xi),\delta_0)$, $\delta_0 = \epsilon e^{-\lambda t}$, and
$$
V_1 = \kappa_\xi \frac{\alpha\delta_0^{p(\xi)}+p(\xi)-\alpha}{\delta_0^{\alpha}\alpha(p(\xi)-\alpha)}  \max\left\{1,M_\xi \left(e^{p(\xi)\lambda t}-1\right)\right\}.
$$
\end{prop}

Let $\varphi_1(u) = f_\xi(u)/(V_1 g_1(u))$ and note that
$$
\varphi_1(u) = \frac{q(\xi,u) - q(\xi,ue^{\lambda t}) }{\left( u^{p(\xi)}1_{[0<u\le \delta_0]} + 1_{[u>\delta_0]}\right)  \max\left\{1,M_\xi \left(e^{p(\xi)\lambda t}-1\right)\right\} }.
$$
With this notation, we get the following rejection sampling algorithm for simulating from $f_\xi$ for a fixed $\xi$. \\

\noindent \textbf{Algorithm 1.}\\
1. Independently simulate $U\sim U(0,1)$ and $Y\sim \mathrm{MLL}(\alpha,p(\xi),\delta_0)$.\\
2. If $U\le \varphi_1(Y)$ return $Y$, otherwise go back to step 1.\\

From general facts about rejection sampling algorithms, we know that, on a given iteration, the probability of accepting the observation is $1/V_1$ and the number of iterations until we accept follows a geometric distribution with mean $V_1$. Thus, the algorithm is more efficient when $V_1$ is small.

\section{$p$-Tempered $\alpha$-Stable Distributions}\label{sec: pTaS}

In the previous sections, we considered tempered stable distributions with very general tempering functions. However, it is often convenient to work with families of tempering functions that have additional structure. One such family, which is commonly used, corresponds to the case where
\begin{eqnarray}\label{eq: q to p temp}
q(\xi,u) = \int_{(0,\infty)} e^{-su^p} Q_\xi(\rd s).
\end{eqnarray}
Here $p>0$ and $\bar Q=\{Q_\xi:\xi\in\mathbb S^{d-1}\}$ is a measurable family of probability measures on $(0,\infty)$. For fixed $p>0$ and $\alpha\in(0,2)$, we refer to the class of tempered stable distributions with such tempering functions as $p$-tempered $\alpha$-stable. For $p=1$, these models were introduced in \cite{Rosinski:2007}, for $p=2$ they were introduced in \cite{Bianchi:Rachev:Kim:Fabozzi:2011}, and the general case was introduced in \cite{Grabchak:2012}. See also the recent monograph \cite{Grabchak:2016book} for an overview.

Now, consider the distribution $\mathrm{TS}_{\alpha,d}(\sigma,q,b)$, where $q$ is as in \eqref{eq: q to p temp} and define the Borel measures
\begin{eqnarray}\label{eq: Q}
Q(B) = \int_{\mathbb S^{d-1}}\int_{(0,\infty)} 1_{B}(s\xi) Q_\xi(\rd s)\sigma(\rd \xi), \ \ B\in\mathfrak B(\mathbb R^d)
\end{eqnarray}
and
\begin{eqnarray}\label{eq: R}
R(B) = \int_{\mathbb R^{d}} 1_{B}\left(\frac{x}{|x|^{1+1/p}}\right) |x|^{\alpha/p}Q(\rd x), \ \ B\in\mathfrak B(\mathbb R^d).
\end{eqnarray}
When we know $R$, we can calculate $Q$ by using the formula
\begin{eqnarray}\label{eq: back to Q}
Q(B) = \int_{\mathbb R^{d}}1_{B}\left( \frac{x}{|x|^{1+p}}\right)|x|^{\alpha} R(\rd x), \ \ B\in\mathfrak B(\mathbb R^d).
\end{eqnarray}
It can be shown that, in this case, the L\'evy measure as given by \eqref{eq: levy ts} can be written as
\begin{eqnarray}\label{eq: L pTS}
L_\alpha(B) = \int_{\mathbb R^{d}}\int_0^\infty 1_{B}(t x) t^{-1-\alpha} e^{-t^p} \rd t R(\rd x), \ \ B\in\mathfrak B(\mathbb R^d)
\end{eqnarray}
and the measure $\sigma$ can be written as
\begin{eqnarray}\label{eq:back to sig}
\sigma(B) = \int_{\mathbb R^d}1_B\left(\frac{x}{|x|}\right)|x|^\alpha R(\rd x),\ \ B\in\mathfrak B(\mathbb S^{d-1}),
\end{eqnarray}
see Chapter 3 in \cite{Grabchak:2016book}. Further, for fixed $\alpha\in(0,2)$ and $p>0$, the measure $R$ uniquely determines both $q$ and $\sigma$. For this reason, we will generally write $\ts(R,b)$ instead of $\mathrm{TS}_{\alpha,d}(\sigma,q,b)$. We call $R$ the Rosi\'nski measure of the distribution. A Borel measure $R$ on $\mathbb R^d$ is the Rosi\'nski measure of some $p$-tempered $\alpha$-stable distribution if and only if
\begin{eqnarray}\label{eq: cond for R measure}
R(\{0\})=0 \mbox{\ and\ } \int_{\mathbb R^d}|x|^\alpha R(\rd x)<\infty.
\end{eqnarray}

\begin{remark}
One can define a more general class of distributions with L\'evy measures of the form \eqref{eq: L pTS}, but where $R$ does not satisfy \eqref{eq: cond for R measure}. In \cite{Grabchak:2016book} distributions that satisfy  \eqref{eq: cond for R measure} are called proper $p$-tempered $\alpha$-stable.
\end{remark}

We now characterize when we can use the representation given in Theorem \ref{thrm: main gen} and when the distributions belong to class F.

\begin{prop}\label{prop: finite K}
Let $\mu=\ts(R,b)$ and let $K$ be as in \eqref{eq:K}. \\
1. If $p\le\alpha$ or $R$ is an infinite measure, then $K=\infty$.\\ 
2. If $p>\alpha$ and $R$ is a finite measure, then
$$
K = R(\mathbb R^d)\alpha^{-1}\Gamma(1-\alpha/p)\left(e^{\lambda t\alpha}-1\right)<\infty
$$
and the result of Theorem \ref{thrm: main gen} holds. If, in addition, 
\begin{eqnarray}\label{eq: on R for F}
\int_{\mathbb R^d} |x|^{-(p-\alpha)} R(\rd x) <\infty,
\end{eqnarray}
then $\mu\in \mathrm{CF}_{\alpha,d}(\epsilon, M_\bullet,p,\sigma)$ with $M_\xi = \int_{(0,\infty)} s Q_\xi(\rd s)$ and any $\epsilon>0$. 
\end{prop}

\begin{remark}
For more on the finiteness of $M_\xi$, as given above, see Lemma \ref{lemma: moments of Q} in Section \ref{sec: proofs} below. When $p\le\alpha$, we have $K=\infty$ and the result of Theorem \ref{thrm: main gen} does not hold. In this case a different framework is needed. Building on the work of \cite{Kawai:Masuda:2012} for certain types of Tweedie distributions, we will develop such a framework in a future work. 
\end{remark}

We now turn to the problem of simulating increments from the transition law of a TSOU-process with limiting measure $\mu=\ts(R,b)$. For the remainder of this section assume that $p>\alpha$, that $0<R(\mathbb R^d)<\infty$, and that \eqref{eq: on R for F} holds. Thus $\mu$ belongs to Class F and we can use Theorem \ref{thrm: main gen}. This requires us to simulate from the distributions $\mathrm{TS}_{\alpha,d}(\sigma_0,q,0)$ and $H(\rd \xi,\rd u)$. It is easy to see that $\mathrm{TS}_{\alpha,d}(\sigma_0,q,0) = \ts(R_0,b)$, where $R_0(\rd x) = (1-e^{-\alpha\lambda t})R(\rd x)$. So long as the support of $\sigma$ contains $d$ linearly independent vectors, we can use Algorithm 1 in \cite{Grabchak:2019} to simulate from this distribution, see Proposition 2 in that paper. To simulate from $H(\rd \xi,\rd u)$ we must simulate from $\sigma_1$ and $f_\xi$. We now derive more explicit formulas for this case.

\begin{prop}\label{prop: form pTaS}
If $\mu=\ts(R,b)$ has $\bar Q=\{Q_\xi:\xi\in\mathbb S^{d-1}\}$, then
\begin{eqnarray*}
\kappa_\xi &=& \frac{\alpha}{\Gamma(1-\alpha/p) \left(e^{\lambda t\alpha}-1\right)  \int_{(0,\infty)} s^{\alpha/p} Q_\xi(\rd s)}= \frac{R(\mathbb R^d)}{K\int_{(0,\infty)} s^{\alpha/p} Q_\xi(\rd s)}, \\
\sigma_1(\rd\xi) &=& \frac{1}{R(\mathbb R^d)}  \int_{(0,\infty)} s^{\alpha/p} Q_\xi(\rd s)\sigma(\rd \xi)
\end{eqnarray*}
and
\begin{eqnarray*}
f_\xi(u) &=&  \kappa_\xi \int_{(0,\infty)}\left(e^{-u^ps} - e^{-u^pse^{p\lambda t}} \right) Q_\xi(\rd s)u^{-1-\alpha}, \ u>0, \ \xi\in\mathbb S^{d-1}.
\end{eqnarray*}
\end{prop}

In light of Proposition \ref{prop: finite K}, we can simulate from $f_\xi$ by using Algorithm 1 with $\delta_0 = \epsilon e^{-\lambda t}=1$, $Y\sim\mathrm{MLL}(\alpha,p,1)$, and 
\begin{eqnarray*}
\varphi_1(u) = \frac{ \int_{(0,\infty)}\left(e^{-u^ps} - e^{-u^pse^{p\lambda t}} \right) Q_\xi(\rd s)}{\left(u^{p}1_{[0<u\le1]} + 1_{[u>1]}\right)\max\left\{1,\int_{(0,\infty)} s Q_\xi(\rd s)\left(e^{p \lambda t}-1\right)\right\}}  .
\end{eqnarray*}
In this case,
\begin{eqnarray*}
V_1 &=& \frac{\max\left\{1, \left(e^{t\lambda p}-1\right)\int_{(0,\infty)} s Q_\xi(\rd s)\right\}}{\Gamma(2-\alpha/p) \left(e^{\lambda t\alpha}-1\right)  \int_{(0,\infty)} s^{\alpha/p} Q_\xi(\rd s)}.
\end{eqnarray*}

In some cases we can improve on Algorithm 1. An issue with the MLL distribution is that it has heavy tails, which can lead to many rejections when the tails of $f_\xi$ are lighter. However, when the support of $Q_\xi$ is lower bounded, we can use the pdf of the generalized gamma distribution, which has lighter tails. The method is based on the following result.

\begin{prop}\label{prop: alt bound}
Let $\zeta= \sup\{c>0: Q_\xi((0,c))=0\}$. If $\zeta>0$ and $\int_{(0,\infty)} s Q_\xi(\rd s)<\infty$, then 
$$
f_\xi(u)\le V_2 g_2(u), \ \ u>0,
$$
where $g_2$ is the pdf of a $\mathrm{GGa}(p-\alpha,p,\zeta)$ distribution and
\begin{eqnarray*}
V_2 &=& \zeta^{\alpha/p-1} \frac{\alpha\left(e^{t\lambda p}-1\right)\int_{(0,\infty)}s Q_\xi(\rd s)}{p\left(e^{t\lambda \alpha}-1\right)\int_{(0,\infty)}s^{\alpha/p} Q_\xi(\rd s)}.
\end{eqnarray*}
\end{prop}

Now, let $\varphi_2(u) = f_\xi(u)/(V_2 g_2(u))$ and note that
$$
\varphi_2(u) = \frac{ \int_{(0,\infty)}\left(e^{-u^ps} - e^{-u^pse^{p\lambda t}} \right) Q_\xi(\rd s) }{\left(e^{t\lambda p}-1\right)\int_{(0,\infty)}s Q_\xi(\rd s) u^p e^{-u^p\zeta}}.
$$
With this notation, when $\zeta>0$, we get the following rejection sampling algorithm for simulating from $f_\xi$ for a fixed $\xi\in\mathbb S^{d-1}$.\\

\noindent \textbf{Algorithm 2.}\\
1. Independently simulate $U\sim U(0,1)$ and $Y\sim \mathrm{GGa}(p-\alpha,p,\zeta)$.\\
2. If $U\le \varphi_2(Y)$ return $Y$, otherwise go back to step 1.\\

Algorithm 2 works better than Algorithm 1 when $V_2<V_1$. Since
$$
\frac{V_2}{V_1} \le  \zeta^{\alpha/p-1} \Gamma(2-\alpha/p)\frac{\alpha}{p},
$$
Algorithm 2 is always better when $\zeta>\left(\Gamma(2-\alpha/p)\frac{\alpha}{p}\right)^{\frac{1}{1-\alpha/p}}$. \\

\noindent\textbf{Example.} In \cite{Kawai:Masuda:2011} a version of Algorithm 2 was derived for one-dimensional TSOU-processes with certain types of Tweedie limiting distributions. Specifically, the distributions considered are of the form $\ts(R,0)$ where $d=1$, $p=1$, $\alpha\in(0,1)$, and $R(\rd x) = a\zeta^\alpha \delta_{1/\zeta}(\rd x)$, where $a,\zeta>0$. These correspond to $\sigma(\rd \xi) =a \delta_{1}(\rd\xi)$ and
$$
q_1(u) = e^{-\zeta u} = \int_{(0,\infty)} e^{-ur}Q_1(\rd r),
$$
where $Q_1(\rd r) = \delta_\zeta(\rd r)$. In this case, the trial distribution is  $\mathrm{GGa}(1-\alpha,1,\zeta)$ and 
\begin{eqnarray*}
\varphi_2(u)= \frac{1-e^{\zeta u(1-e^{\lambda t})}}{\zeta u(e^{\lambda t}-1)}.
\end{eqnarray*}
It follows that our Algorithm 2 reduces to Algorithm 3 in \cite{Kawai:Masuda:2011}. We note that, in that paper, there appears to be a typo in the formula for $\varphi_2$, which they denote by $g_{2,\Delta}$. It should be as given above.

\section{Simulation Study}\label{sec: sims}

In this section we perform a small-scale simulation study to see how well our methodology works in practice. We focus on a family of tempered stable distributions for which the transition law had not been previously derived. 

\subsection{Power Tempered Stable Distributions}

A distribution is said to be power tempered stable if it is  $\ts(R,b)$ with $d=1$, $p=1$, $\alpha\in(0,1)$, and
$$
R(\rd x) =A(1+|x|)^{-2-\alpha-\ell} \rd x,
$$
where 
$$
A =  c(\alpha+\ell+1)\frac{\alpha}{\Gamma(1-\alpha)}
$$
and $\ell,c>0$ are parameters. We denote such distributions by $\mathrm{PT}_\alpha(\ell,c)$. These models were introduced in  \cite{Grabchak:2016book} as a class of tempered stable distributions with a finite mean, but still fairly heavy tails. In fact, if $Y\sim\mathrm{PT}_\alpha(\ell,c)$, then, for $\beta\ge0$,
$$
\rE|Y|^\beta<\infty \mbox{ if and only if } \beta<1+\alpha+\ell.
$$
Thus, $\ell$ controls how heavy the tails are. Methods for evaluating the pdfs and related quantities for these distribution are available in the SymTS package \cite{Grabchak:Cao:2017} for the statistical software R.  We now give some useful facts.

\begin{prop}\label{prop: for example 1}
For a $\mathrm{PT}_\alpha(\ell,c)$ distribution with $\alpha\in(0,1)$ and $\ell,c>0$ we have
$$
R(\mathbb R) = c \frac{2\alpha}{\Gamma(1-\alpha)}<\infty \ \mbox{  and  } \int_{\mathbb R} |x|^{-(1-\alpha)} R(\rd x) <\infty.
$$
Further, 
$$
Q(\rd x) = A (1+|x|)^{-2-\alpha-\ell}  |x|^{\ell}\rd x,
$$
\begin{eqnarray*}
\sigma(\rd \xi) = AB \big(\delta_{-1}(\rd \xi) + \delta_{1}(\rd \xi) \big),
\end{eqnarray*}
and for $\xi\in\mathbb S^0=\{-1,1\}$
\begin{eqnarray}\label{eq: Q xi for power}
Q_\xi(\rd x) = \frac{1}{B} (1+x)^{-2-\alpha-\ell}  x^{\ell}1_{[x>0]}\rd x, 
\end{eqnarray}
where
\begin{eqnarray*}
B = \int_{0}^\infty (1+x)^{-2-\alpha-\ell} x^{\ell}\rd x=\int_{0}^\infty (1+x)^{-2-\alpha-\ell} x^{\alpha}\rd x.
\end{eqnarray*}
\end{prop}

For these distributions, an exact simulation technique, based on rejection sampling, was developed in \cite{Grabchak:2019}. Specifically, to simulate from $\mathrm{PT}_\alpha(\ell,c)$ we can use Algorithm 1 in \cite{Grabchak:2019} with, in the notation of that paper,
$$
\eta = \frac{\Gamma(1-\alpha)}{\alpha}R(\mathbb R) = 2c,
$$
see Proposition 2 in \cite{Grabchak:2019}. Detailed simulations can be found in that paper.

Here, we are interested in simulating increments from a TSOU-process with limiting distribution $\mathrm{PT}_\alpha(\ell,c)$. Proposition \ref{prop: finite K} implies that this distribution belongs to Class F and that we can use the representation given in Theorem \ref{thrm: main gen}. To use this representation, we need to simulate from $\mathrm{TS}^1_{\alpha,1}(R_0,0)$ and $H(\rd \xi,\rd u)$. It is readily checked that, in this case, $\mathrm{TS}^1_{\alpha,1}(R_0,0)=\mathrm{PT}_\alpha(\ell,\left(1-e^{-\alpha\lambda t}\right)c)$ and, thus, that we can use the methodology described above to simulate this component. Further, to implement our methodology, we can use the following.

\begin{prop}\label{prop: for example}
For a  $\mathrm{PT}_\alpha(\ell,c)$ distribution with $\alpha\in(0,1)$ and $\ell,c>0$ we have
$$
\sigma_1(\rd \xi) = .5\big(\delta_{-1}(\rd \xi) + \delta_{1}(\rd \xi) \big)
$$
and 
$$
V_1 = \frac{(\alpha+\ell+1)B}{\Gamma(2-\alpha)(e^{\lambda t\alpha}-1)}\max\left\{1, \frac{e^{t\lambda}-1}{B} \int_0^\infty (1+s)^{-2-\alpha-\ell} s^{\ell+1} \rd s \right\} .
$$
Further, for $\xi\in\mathbb S^0=\{-1,1\}$ and $u>0$ we have
\begin{eqnarray}\label{eq: f xi for power}
f_\xi(u) = \frac{\alpha (\alpha+\ell+1)}{\Gamma(1-\alpha)(e^{\lambda t\alpha}-1)} \int_0^\infty \left(e^{-us}-e^{-use^{\lambda t}}\right) (1+s)^{-2-\alpha-\ell} s^{\ell}\rd s u^{-1-\alpha},
\end{eqnarray}
and for $u>0$ we have
\begin{eqnarray*}
\varphi_1(u) = \frac{ \int_0^\infty \left(e^{-us} - e^{-use^{\lambda t}} \right) (1+s)^{-2-\alpha-\ell}s^\ell \rd s}{B\left(u 1_{[0<u\le1]} + 1_{[u>1]}\right)\max\left\{1, \frac{e^{t\lambda}-1}{B} \int_0^\infty (1+s)^{-2-\alpha-\ell} s^{\ell+1} \rd s \right\}}.
\end{eqnarray*}
\end{prop}

Note that none of the quantities discussed in Proposition \ref{prop: for example} depend on the parameter $c$. Further, both $f_\xi$ and $\varphi_1$ do not depend on $\xi$. Now, we can simulate from $H$ as follows. First, we simulate $\xi$ from $\sigma_1$ by taking
\begin{eqnarray}\label{eq: xi example}
\xi = \left\{\begin{array}{rl}
-1 & \mbox {with probability } .5\\
1 & \mbox {with probability }  .5
\end{array}
\right..
\end{eqnarray}
Next, we simulate $W$ from $f_\xi$. We do this using Algorithm 1, in which we take $Y\sim \mathrm{MLL}(\alpha,1,1)$ and $\varphi_1$ as given above.

\subsection{Univariate Simulation Results}

We now present the results of our simulation study, where we focused on simulating increments from a TSOU-process with parameter $\lambda$ and limiting distribution $\mathrm{PT}_\alpha(\ell,c)$. For simplicity, throughout this section, we fix the parameters $\lambda=1$ and $c=1$. First, we consider simulation from $H$. Here we choose the parameter values $\alpha=0.55$, $\ell=1$, and $t=0.1$. Figure 1 gives a plot of the pdf $f_1$ (solid line) overlaid with a plot of the pdf of $\mathrm{MLL}(0.55,1,1)$ (dashed line). We then simulate $n=10^7$ observations from $\mathrm{MLL}(0.55,1,1)$, and, using Algorithm 1, we decide which observations should be rejected. We numerically calculated that $V_1=12.8837$ and thus we  expect to get $n/V_1=776175.4$ observations. We actually obtained $776528$ observations. We then assigned to each of these a positive or a negative sign based on \eqref{eq: xi example}. A plot of the kernel density estimate (KDE) of the pdf of the resulting observations is given in Figure 2. This is overlaid (dashed line) with the pdf of the true density of $X=\xi W$. 

\begin{figure}\label{fig: true pdf}
\linespread{.75}
\begin{tabular}{cc}
\includegraphics[trim={1.15cm 1cm 1cm .5cm},clip,scale=.3]{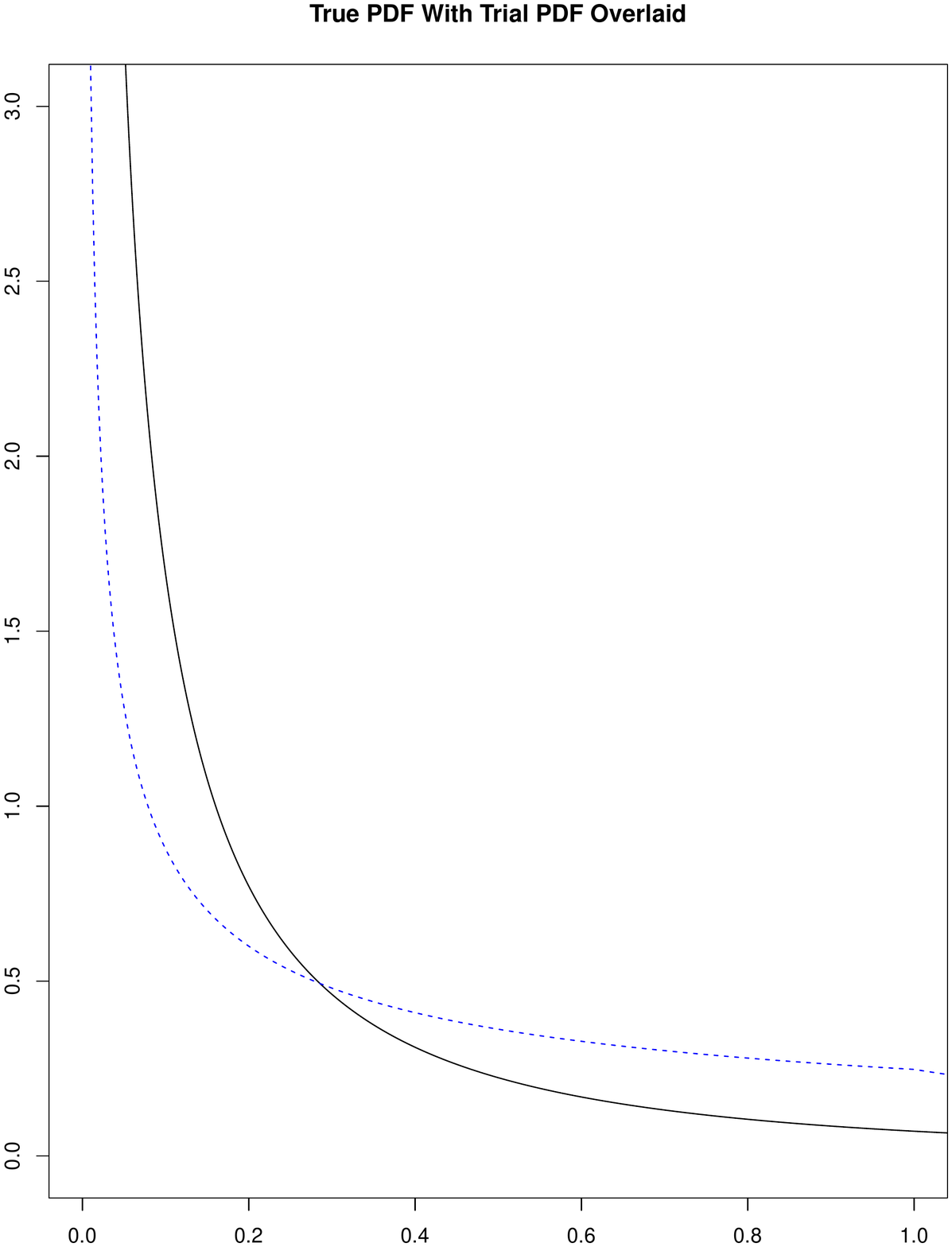} & \includegraphics[trim={1.15cm 1cm 1cm .5cm},clip,scale=.3]{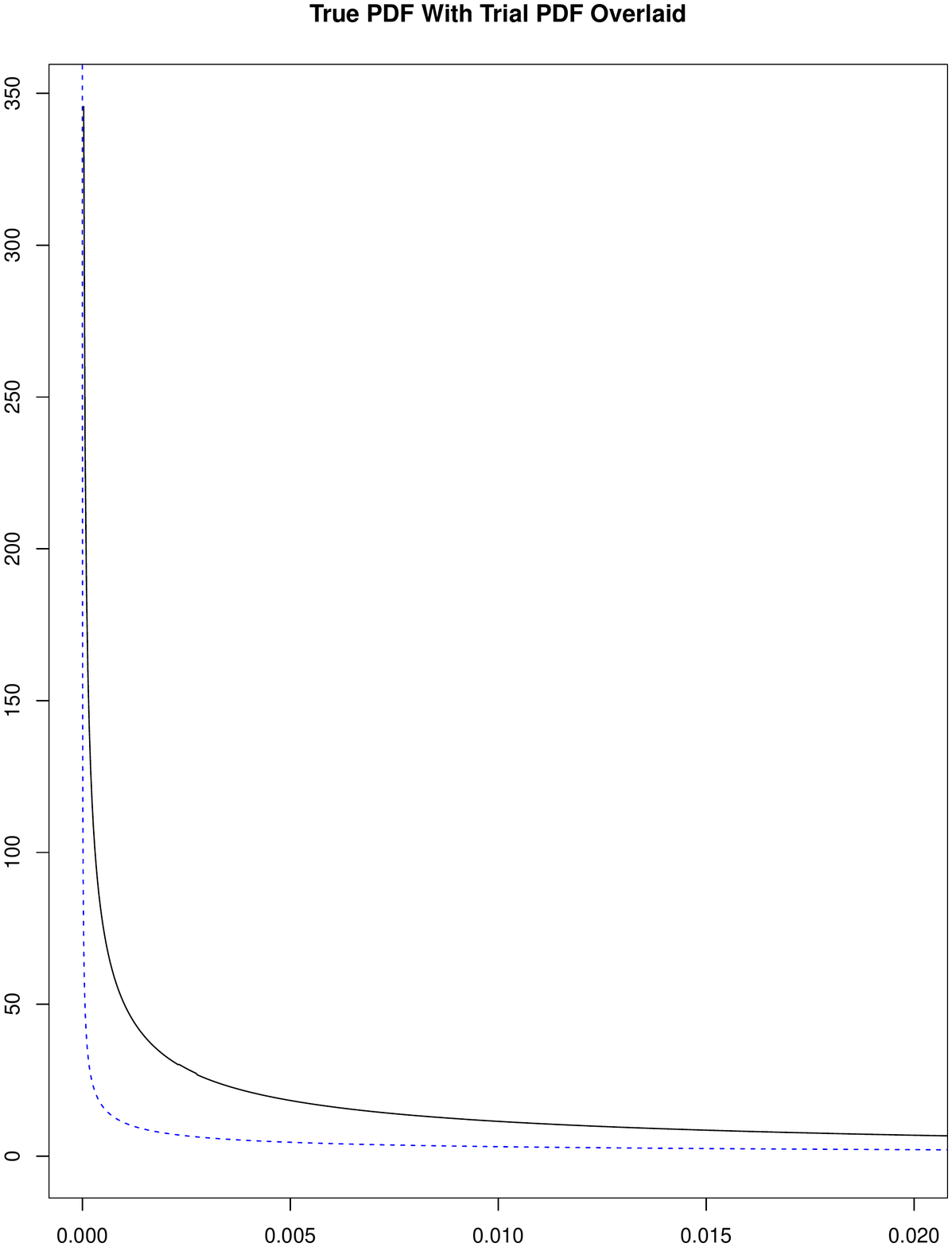} 
\end{tabular}
\vspace{-.5cm}\caption{For the parameter values $\alpha=0.55$, $t=0.1$ and $\ell=1$, we plot the pdf $f_1$ (solid line) with the pdf of $\mathrm{MLL}(0.55,1,1)$ (dashed line) overlaid. These are presented at two scales.}
\end{figure}

\begin{figure}\label{fig: estimated densities}
\linespread{.75}
\begin{tabular}{cc}
\includegraphics[trim={1.15cm 1cm 1cm .5cm},clip,scale=.3]{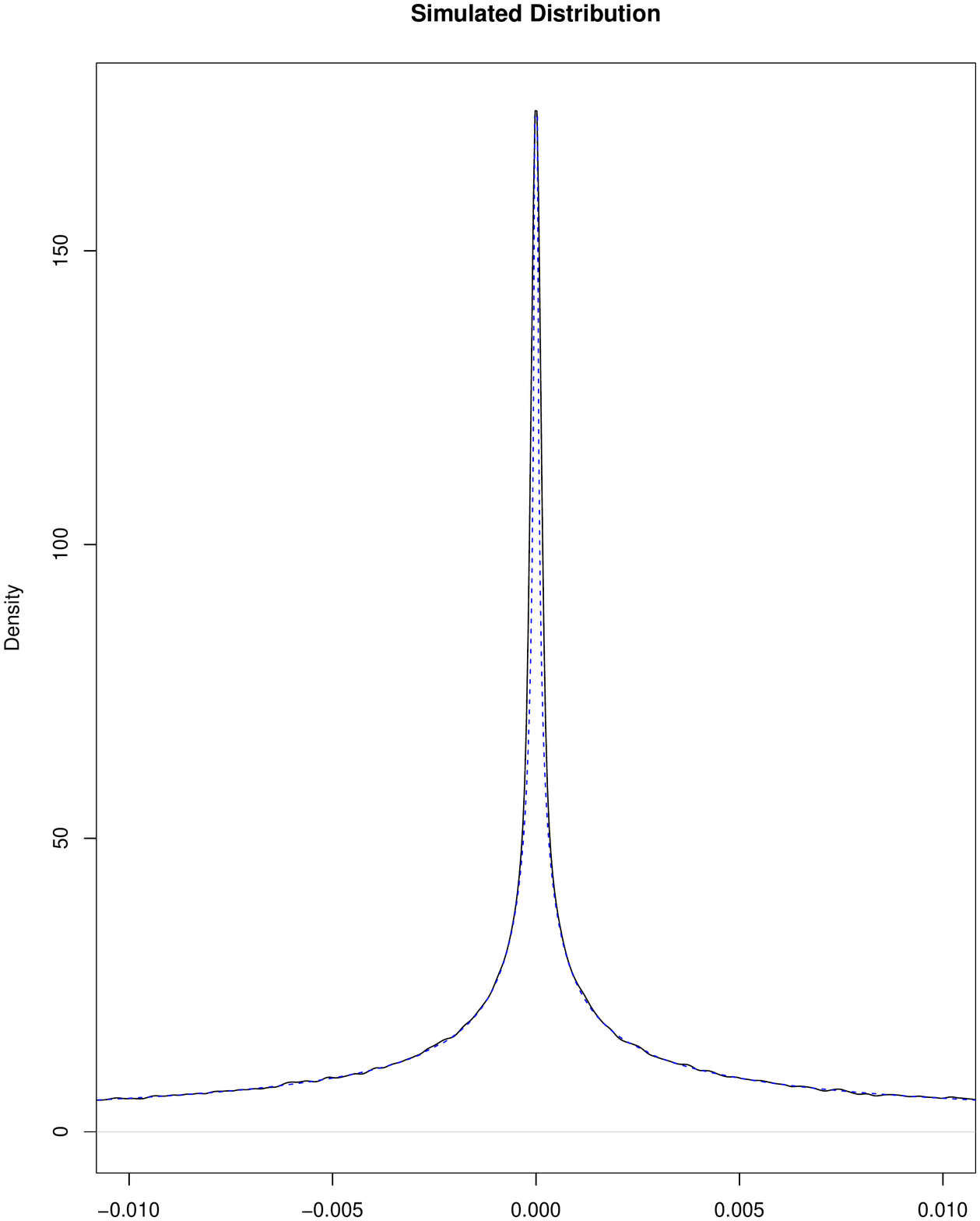}
& \includegraphics[trim={1.15cm 1cm 1cm .5cm},clip,scale=.3]{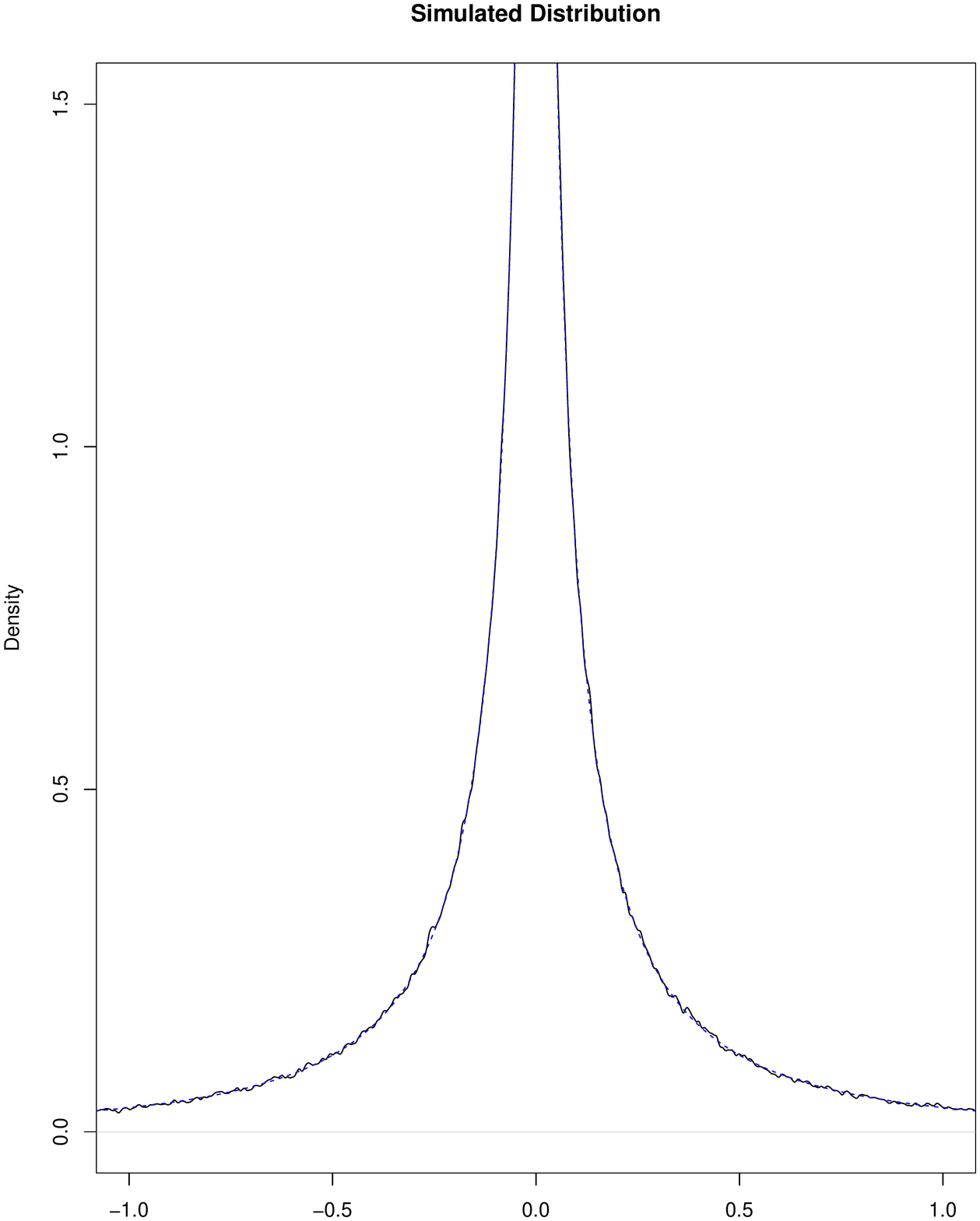}
\end{tabular}
\vspace{-.5cm}\caption{For the parameter values $\alpha=0.55$, $t=0.1$, and $\ell=1$, we used Algorithm 1 and \eqref{eq: xi example} to simulate $776528$ observations from  the distribution of $X=\xi W$. We plot the KDE of their density (solid line) with the true pdf overlaid (dashed line). These are presented at two scales. }
\end{figure}

We now turn to the simulation of TSOU-processes.  Figure 3 presents four simulated paths of the TSOU-process with limiting distribution $\mathrm{PT}_\alpha(\ell,1)$ for different choices of the parameters. The paths go from time $T=0$ to time $T=100$ in increments of $t=0.1$. Thus, each path consists of $1000$ increments. For simplicity, we start each path at $0$, which is the mean of the limiting distribution. To check that we really get the correct limiting distribution, we simulated $3000$ paths of the process for two choices of the parameters. We then took the final observation (at time $T=100$) from each of these. These are independent random variables from $\mathrm{PT}_\alpha(\ell,1)$. In Figure 4 we plot the KDE of their densities (solid line) with the true pdf of $\mathrm{PT}_\alpha(\ell,1)$ overlaid (dashed line).

\begin{figure}\label{fig: TSOU}
\linespread{.75}
\begin{tabular}{cc}
\includegraphics[trim={1.15cm 1cm 1cm .5cm},clip,scale=.3]{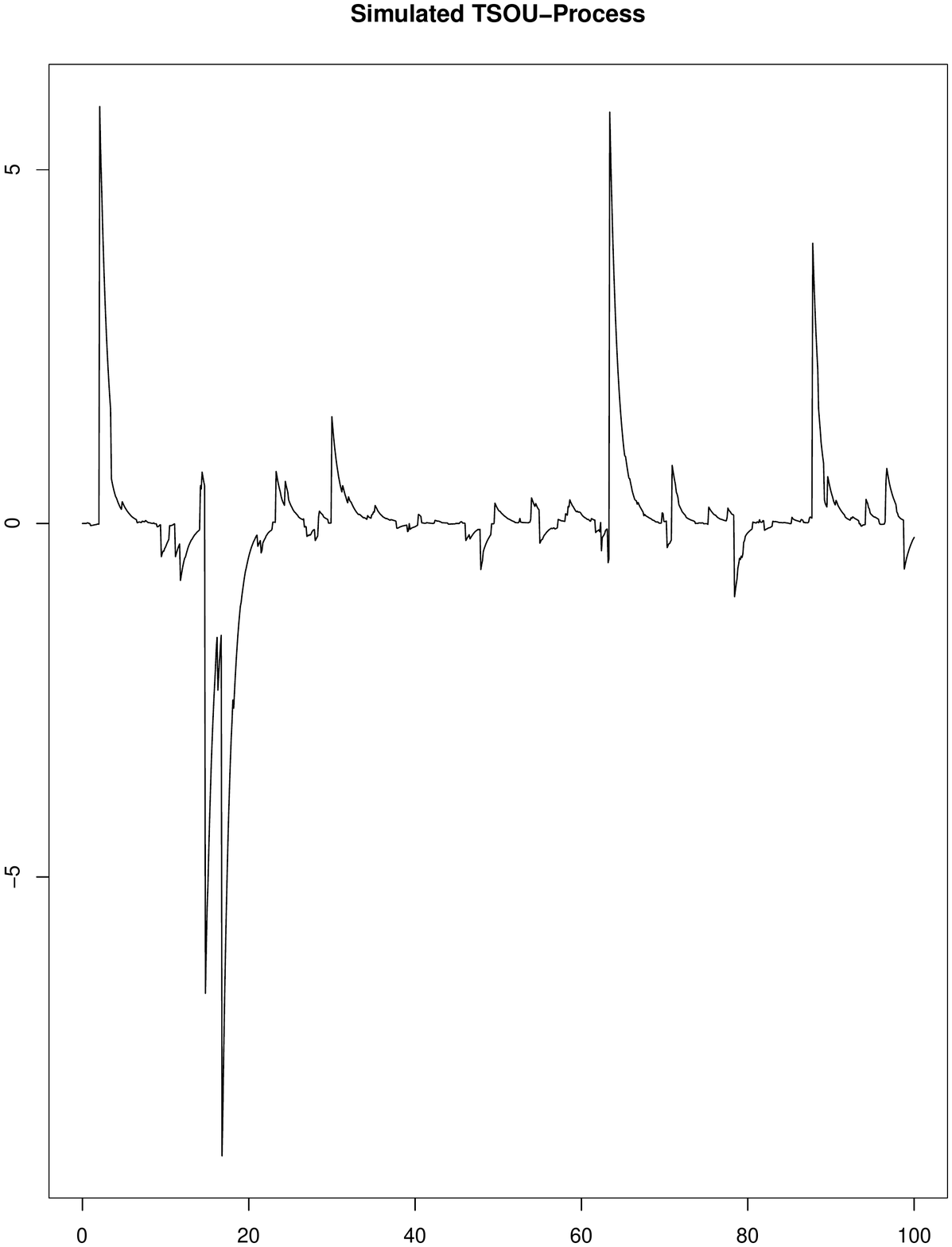} & \includegraphics[trim={1.15cm 1cm 1cm .5cm},clip,scale=.3]{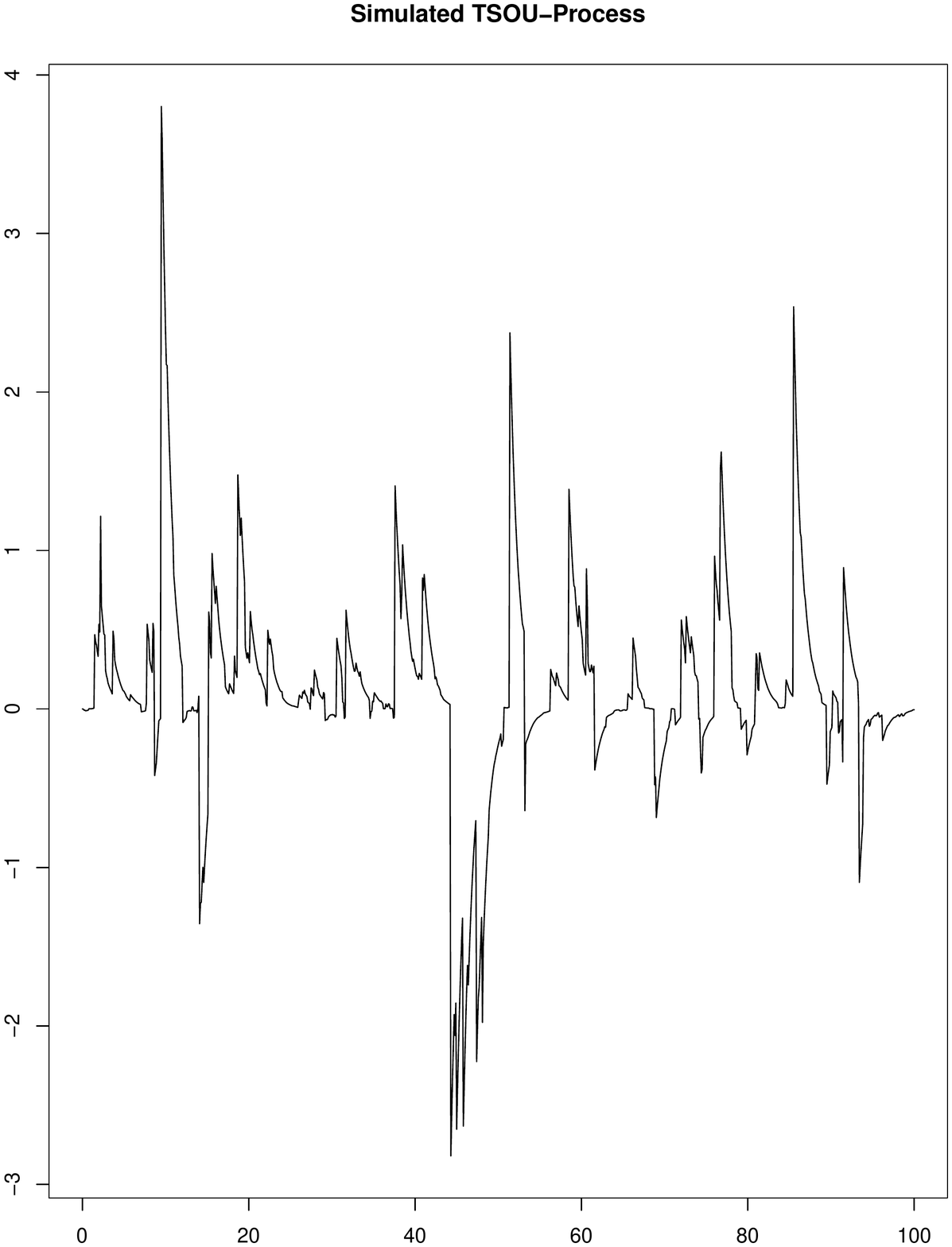} \\
$\alpha=.55,\ell=1$ & $\alpha=.55,\ell=10$ \vspace{.4cm}\\
\includegraphics[trim={1.15cm 1cm 1cm .5cm},clip,scale=.3]{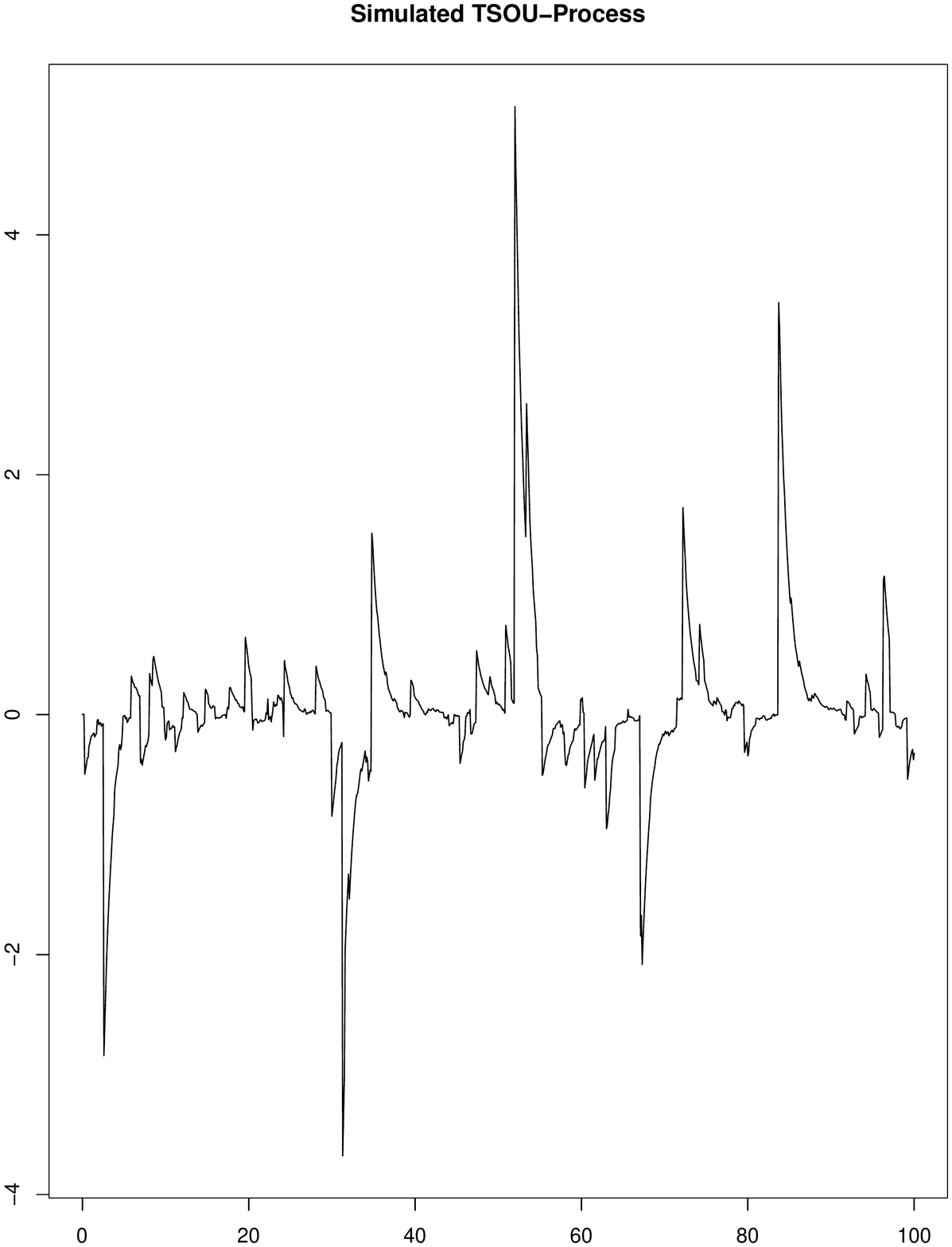} & \includegraphics[trim={1.15cm 1cm 1cm .5cm},clip,scale=.3]{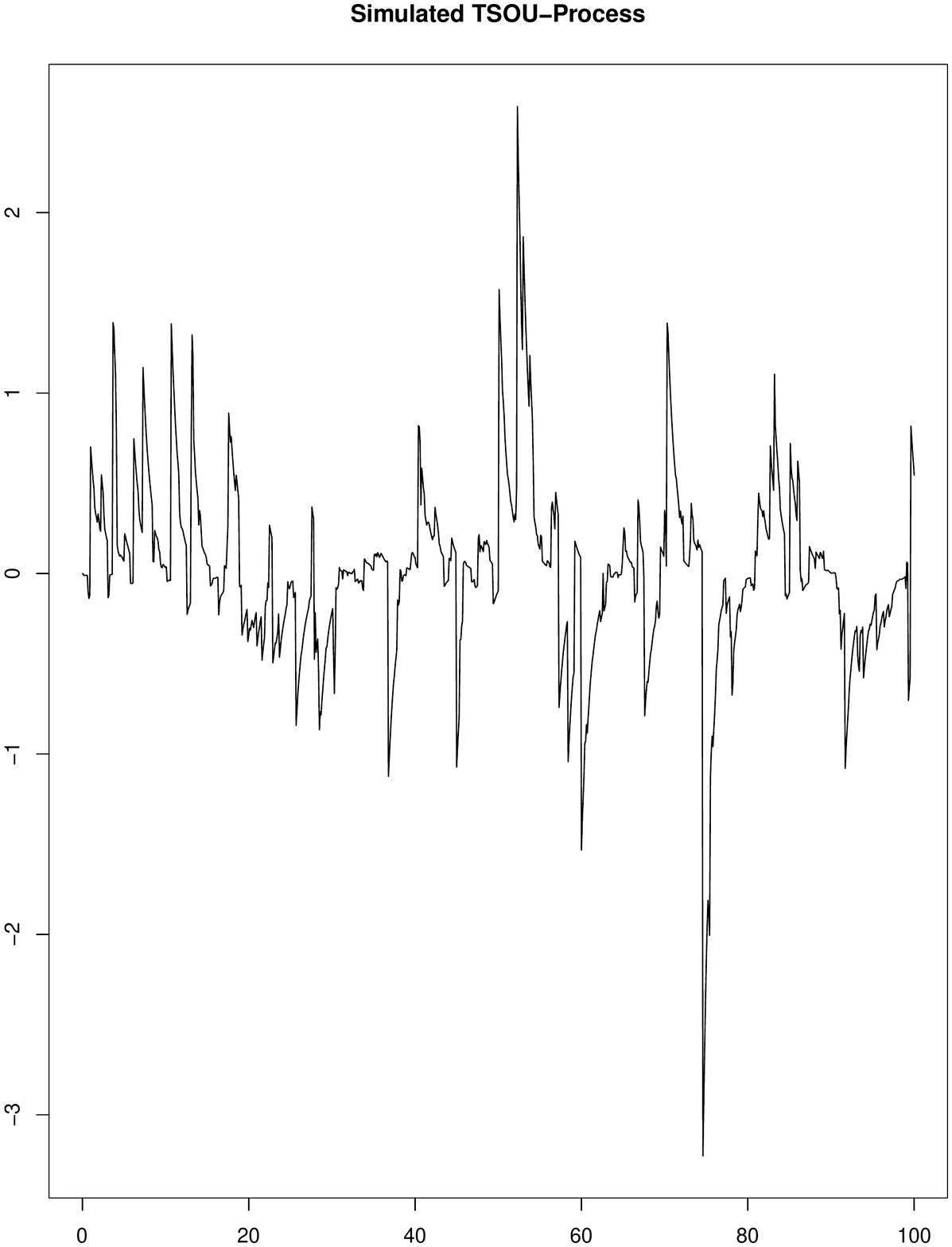} \\
$\alpha=.75,\ell=1$ & $\alpha=.75,\ell=10$\vspace{.4cm}
\end{tabular}
\vspace{-.5cm}\caption{Simulated TSOU-processes for several choices of the parameters. In all cases the simulated increments were of length $t=0.1$.}
\end{figure}

\begin{figure}\label{fig: limit}
\linespread{.75}
\begin{tabular}{cc}
\includegraphics[trim={1.15cm 1cm 1cm .5cm},clip,scale=.375]{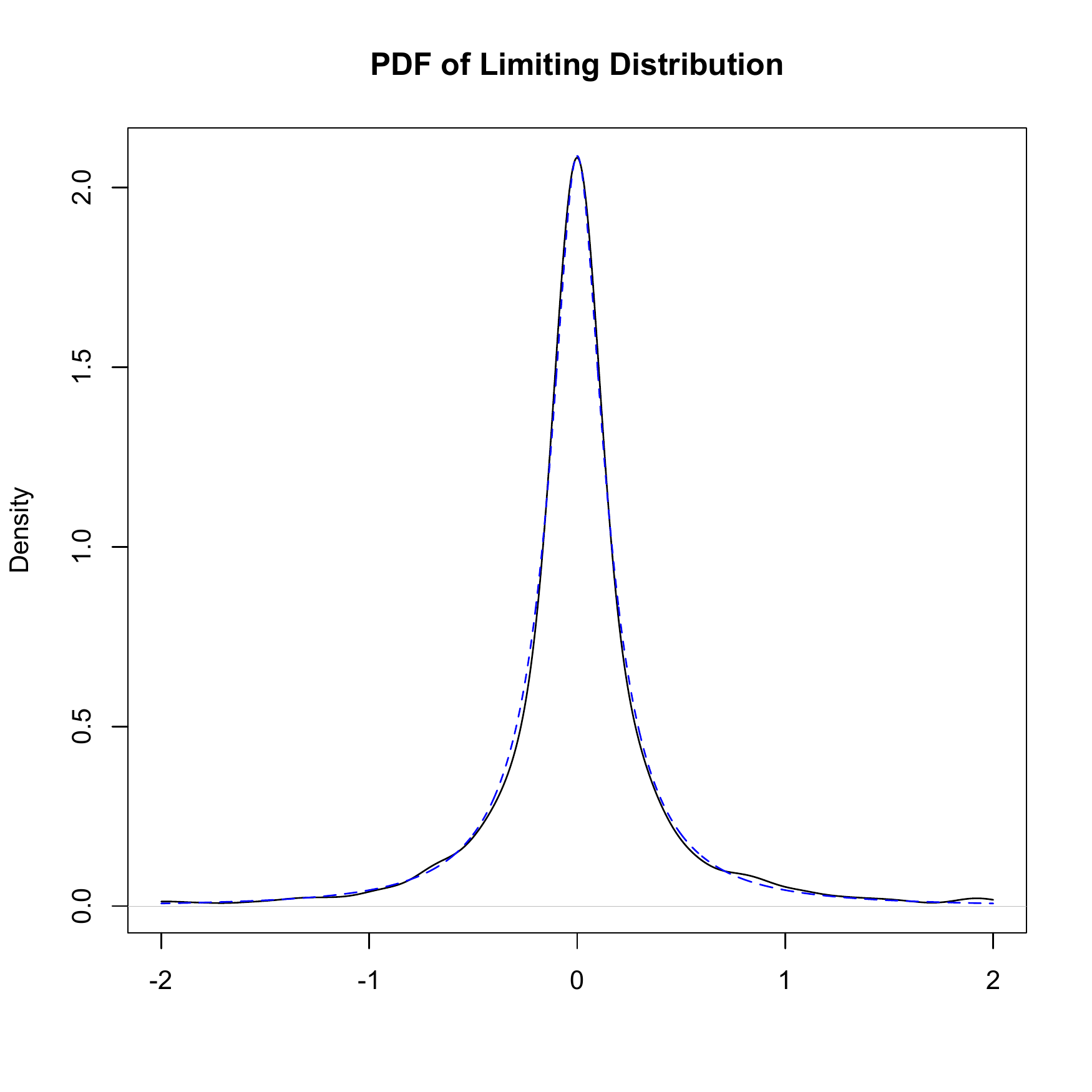} & \includegraphics[trim={1.15cm 1cm 1cm .5cm},clip,scale=.375]{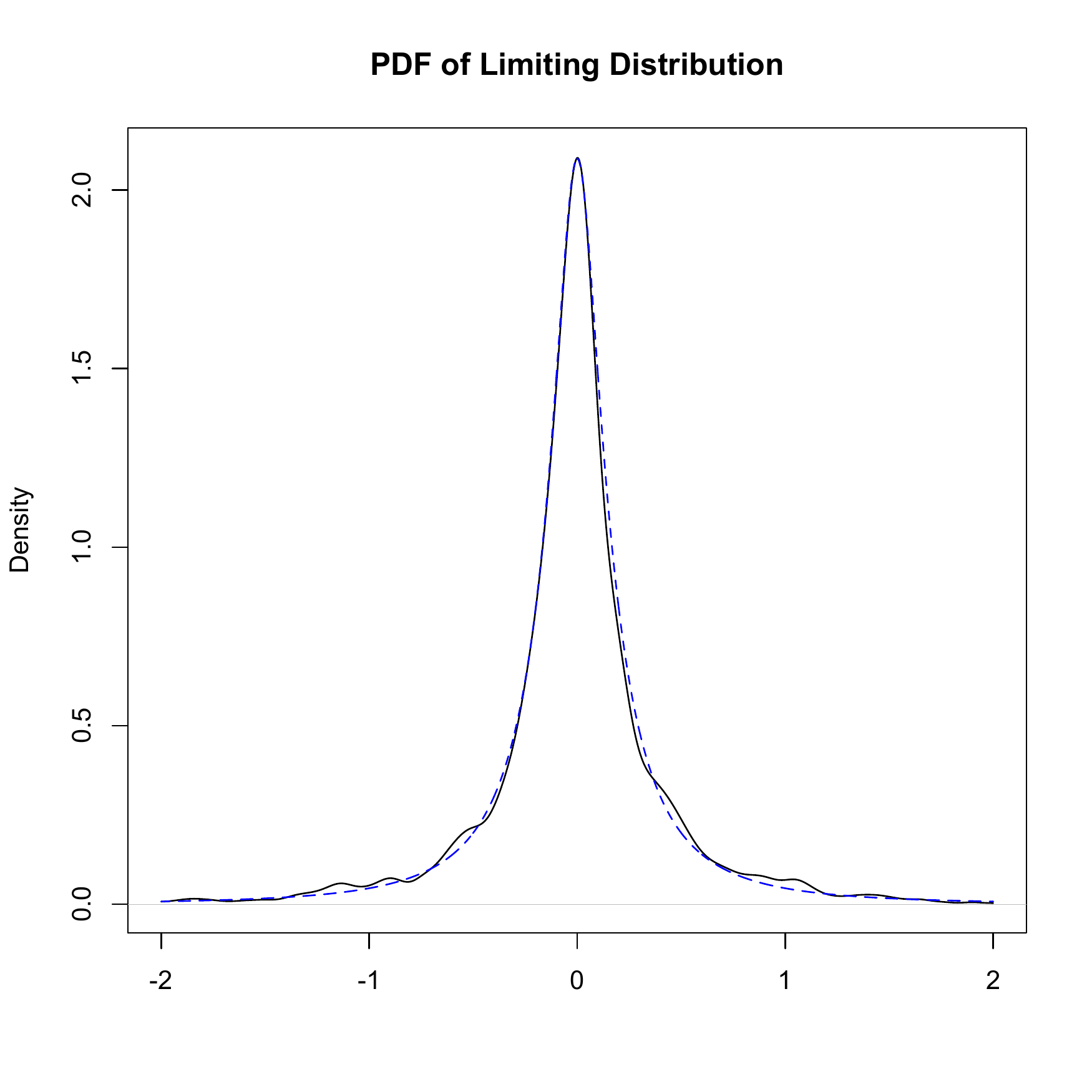}\\
$\alpha=.55,\ell=1$ & $\alpha=.75,\ell=1$ \vspace{.4cm}
\end{tabular}
\vspace{-.5cm}\caption{We simulated $3000$ TSOU-processes (with the same parameters) up to time $T=100$ in increments of $t=0.1$. We then consider the last observation for each process. We plot the KDE these observations (solid line) overlaid with the true pdf of the limiting distribution $\mathrm{PT}_\alpha(\ell,1)$ (dashed line).}
\end{figure}

\subsection{Multivariate Simulation Results}

In this section we give simulations for the multivariate case. We begin by introducing a multivariate extension of the power tempered stable distribution with a finite measure $\sigma$. Let $d\ge2$ be the dimension, let $s_1, s_2,\dots, s_k\in\mathbb S^{d-1}$ with $s_i\ne s_j$ and $s_i\ne -s_j$ for $i\ne j$, and let
\begin{eqnarray}\label{eq: lin comb PT}
X = \sum_{j=1}^k X_i s_j,
\end{eqnarray}
where $X_1,X_2,\dots,X_k\iid \mathrm{PT}_\alpha(\ell,c)$. It is not difficult to show that $X\sim \mathrm{TS}^1_{\alpha,d}(R,0)$ with
 $$
 R(B) = \int_{\mathbb S^{d-1}}\int_0^\infty 1_B(u \xi) (u+1)^{-2-\alpha-\ell} \rd u \sigma(\rd \xi), \ \ B\in\mathcal B(\mathbb R^d),
  $$
 where
$$
 \sigma(\rd \xi) = \sum_{s\in S} \delta_{s}(\rd\xi)
 $$
 and $S=\{\pm s_1,\pm s_2,\dots,\pm s_k\}$.  We denote this distribution by $\mathrm{PT}_{\alpha,d,k}(\ell,c,S)$. Further, we can show that, in this case,
$$
\sigma_1(\rd \xi) = \frac{1}{2k} \sum_{s\in S} \delta_{s}(\rd\xi),
$$
and that, for $\xi\in S$, $Q_\xi$ is given by \eqref{eq: Q xi for power} and $f_\xi$ is given by \eqref{eq: f xi for power}.

We are interested in simulating a multivariate TSOU-process with parameter $\lambda>0$ and limiting distribution $\mathrm{PT}_{\alpha,d,k}(\ell,c,S)$. Proposition \ref{prop: finite K} implies that this distribution belongs to Class F and that we can use the representation of the transition law given in \eqref{eq: trans RV}. To use this representation, we need to simulate from $\mathrm{TS}^1_{\alpha,d}(R_0,0)$ and $H(\rd \xi,\rd u)$. It is readily checked that $\mathrm{TS}^1_{\alpha,d}(R_0,0)=\mathrm{PT}_{\alpha,d,k}(\ell,\left(1-e^{-\alpha\lambda t}\right)c,S)$ and (since we already know how to simulate from $\mathrm{PT}_\alpha(\ell,\left(1-e^{-\alpha\lambda t}\right)c)$) we can use \eqref{eq: lin comb PT} to simulate from this distribution. To simulate from $H$, we first simulate $\xi$ uniformly from the finite set $S$, then we simulate $W$ from $f_\xi$.  For simplicity, in our simulations, we take $d=2$, $k=3$, and we fix the parameters $\lambda=1$, $c=1$, $\ell=1$, and $\alpha=0.55$. We simulated two TSOU-processes with different choices of $S$.  As in the univariate case, the paths go from time $T=0$ to time $T=100$ in increments of $t=0.1$. For simplicity we start each path at $\left({0\atop0}\right)$, which is the mean of the limiting distribution.  The plots of the $x$ and $y$ coordinates of the two processes are given in Figure 5.

\begin{figure}\label{fig: TSOU Multivariate}
\linespread{.75}
\begin{tabular}{cc}
\includegraphics[trim={1.15cm 1cm 1cm .5cm},clip,scale=.375]{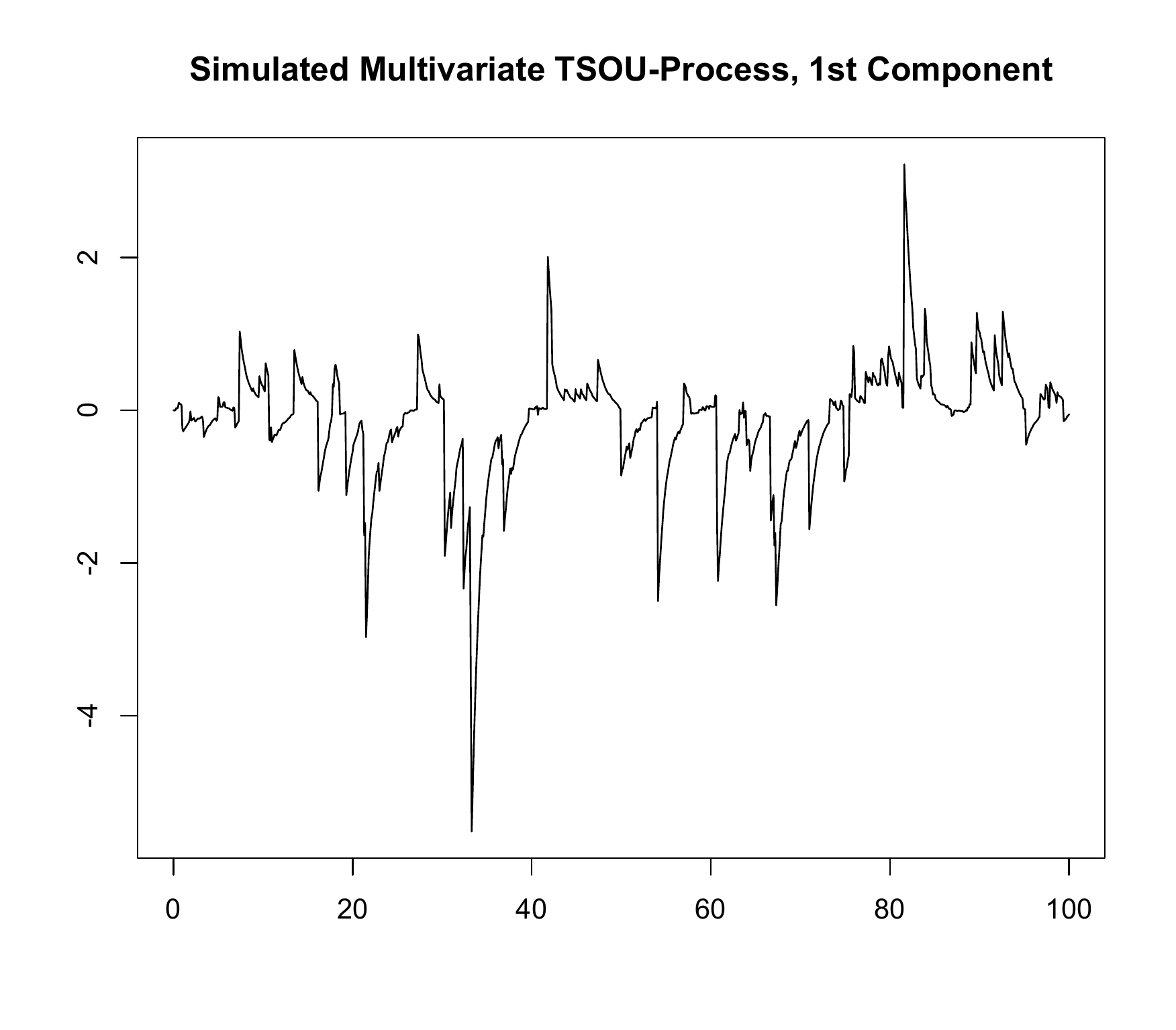} & \includegraphics[trim={1.15cm 1cm 1cm .5cm},clip,scale=.375]{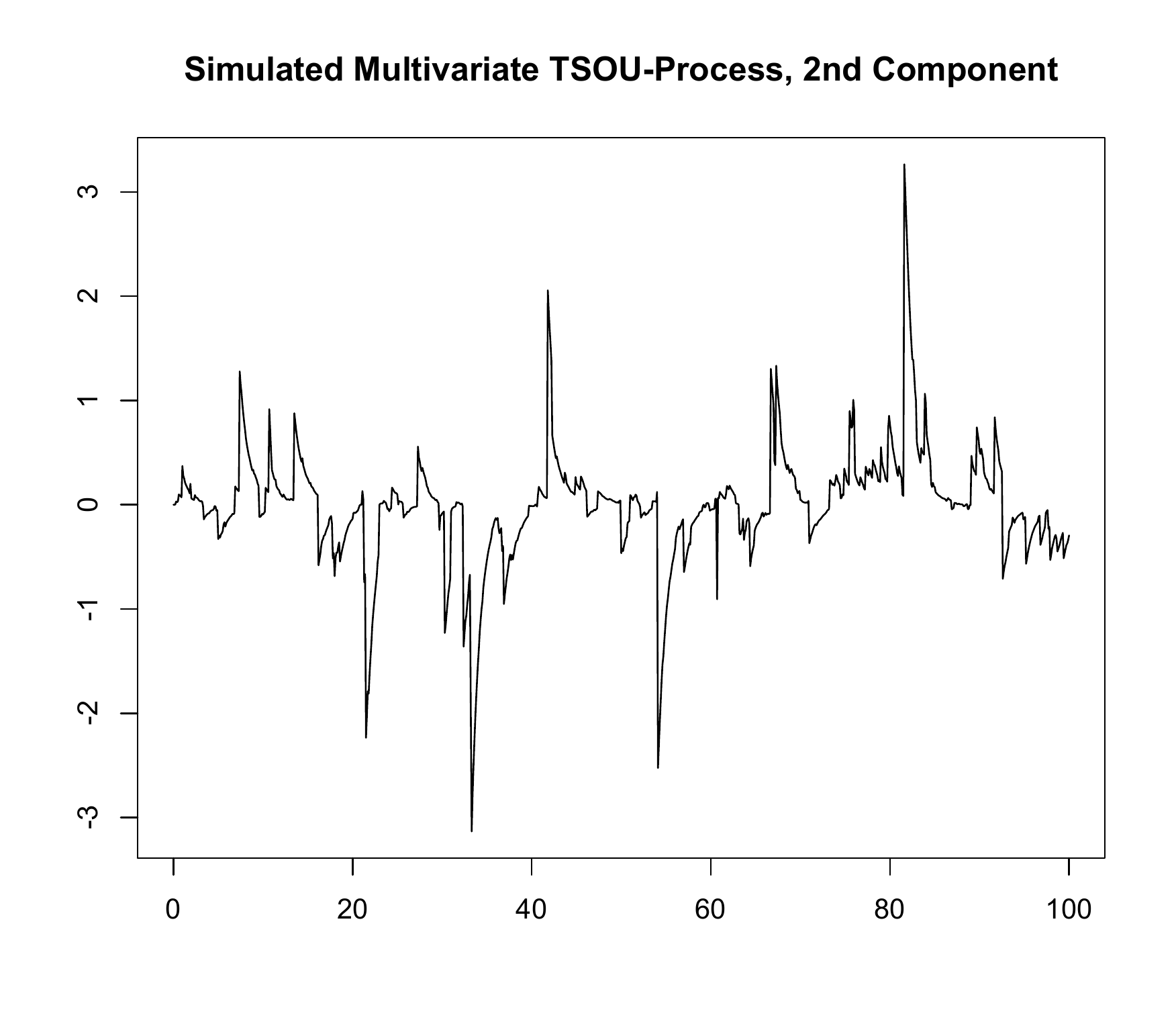} \\
\multicolumn{2}{c}{$s_1 =\left( {\sqrt 2/2\atop \sqrt 2/2}\right)$, $s_2=\left({\sqrt 3/2 \atop 1/2}\right)$, $s_3 = \left({-\sqrt 2/2\atop \sqrt 2/2}\right)$} \vspace{.4cm}\\
\includegraphics[trim={1.15cm 1cm 1cm .5cm},clip,scale=.375]{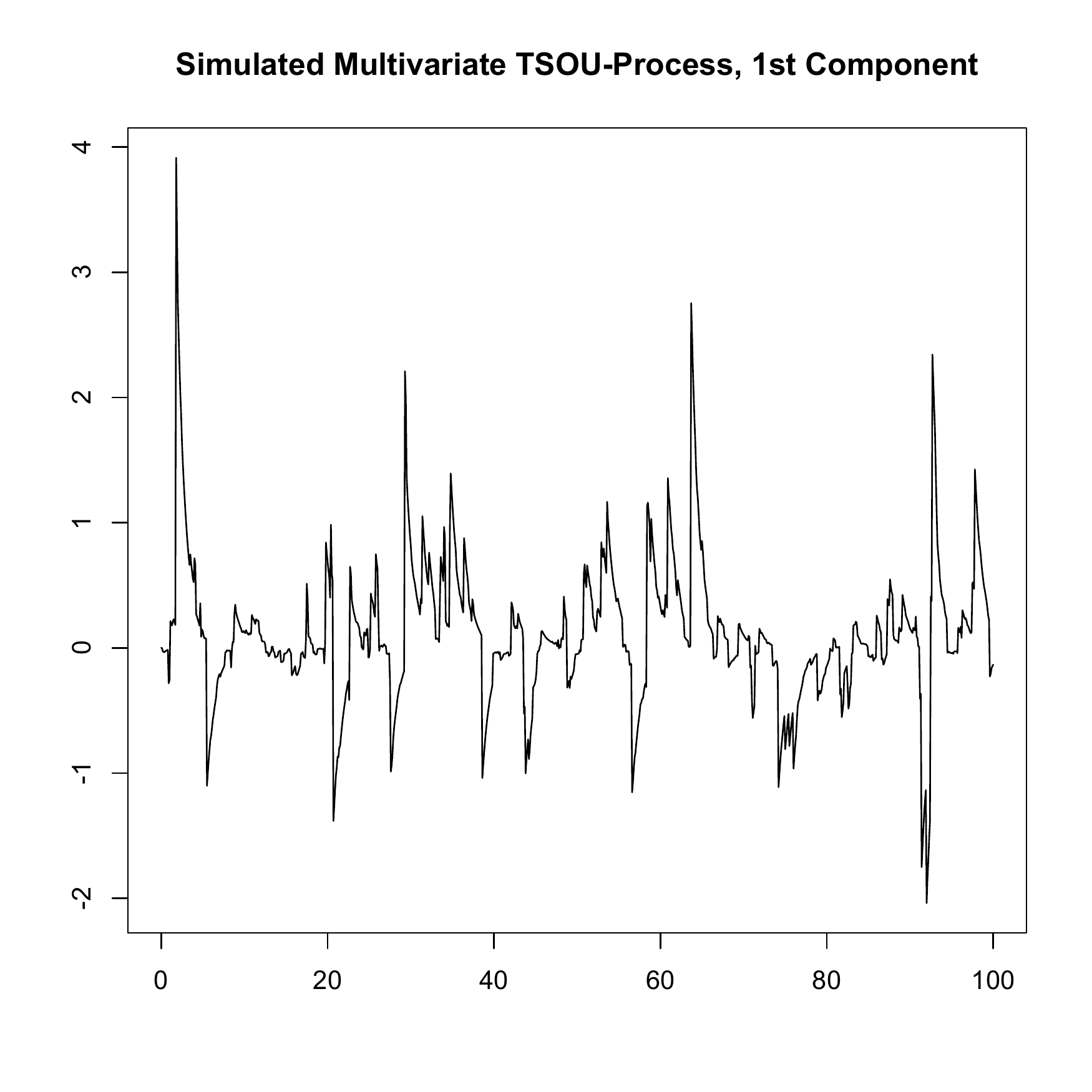} & \includegraphics[trim={1.15cm 1cm 1cm .5cm},clip,scale=.375]{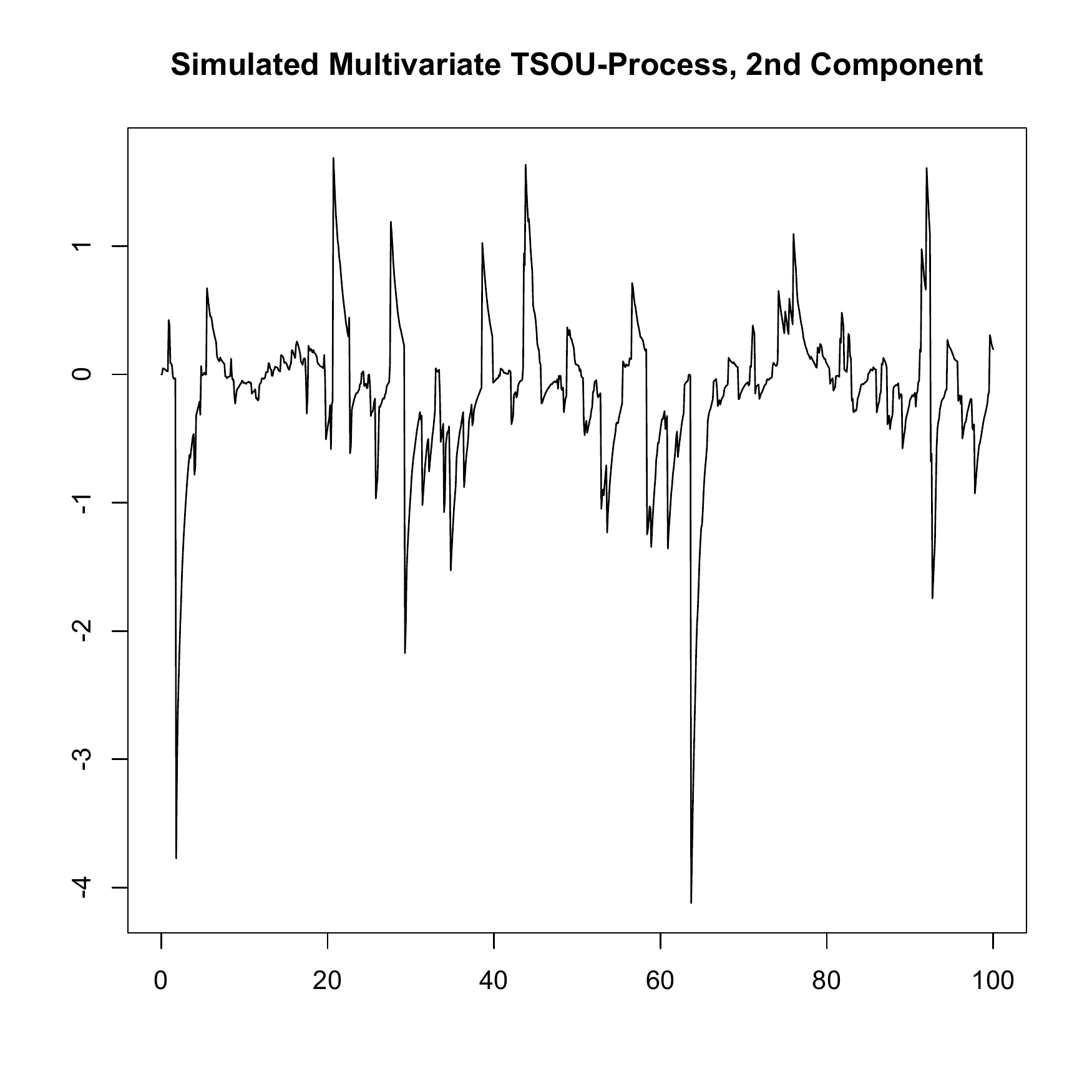} \\
\multicolumn{2}{c}{$s_1 =\left( {\sqrt 2/2\atop -\sqrt 2/2}\right)$, $s_2=\left({\sqrt 3/2 \atop -1/2}\right)$, $s_3 = \left({2/\sqrt{13} \atop -3/\sqrt{13}}\right)$} \vspace{.4cm}\\
\end{tabular}
\vspace{-.5cm}\caption{Simulated multivariate TSOU-processes for two choices of $S$. On the left is the $x$-coordinate of the process and on the right is the $y$-coordinate.}
\end{figure}

\section{Proofs}\label{sec: proofs}

We begin with several lemmas.

\begin{lemma}\label{lemma: moments of Q}
Let $Q$ be given by \eqref{eq: Q} and $R$ by \eqref{eq: R}. For any $\gamma>0$, we have
\begin{eqnarray*}
\int_{\mathbb R^d} |x|^{\alpha-\gamma p} R(\rd x) <\infty
\end{eqnarray*}
if and only if 
\begin{eqnarray*}
\int_{\mathbb S^{d-1}} \int_{(0,\infty)} s^\gamma Q_\xi(\rd s) \sigma(\rd \xi) <\infty.
\end{eqnarray*}
\end{lemma}

\begin{proof}
The fact that
$$
\int_{\mathbb R^d} |x|^{\alpha-\gamma p} R(\rd x) = \int_{\mathbb R^d} |x|^\gamma Q(\rd x) = \int_{\mathbb S^{d-1}} \int_{(0,\infty)}s^\gamma Q_\xi(\rd s)\sigma(\rd\xi)
$$
gives the result.
\end{proof}

\begin{lemma}\label{lemma: exp fact}
If $0\le a<b<\infty$ then
$$
0< e^{-b}(b-a) \le e^{-a} - e^{-b}\le e^{-a}(b-a)\le(b-a).
$$
\end{lemma}

\begin{proof}
The result follows from the fact that $e^{-a} - e^{-b} = \int_a^b e^{-u}\rd u$.
\end{proof}

\begin{lemma}\label{lemma: integ diff}
For any $\kappa>0$ and $0<\alpha<p$
$$
 \int_0^\infty \left(e^{-u^p}-e^{-u^{p}\kappa}\right) u^{-1-\alpha} \rd u = \alpha^{-1}\Gamma(1-\alpha/p)\left(\kappa^{\alpha/p}-1\right).
 $$
\end{lemma}

\begin{proof}
We have
 \begin{eqnarray*}
 \int_0^\infty \left(e^{-u^p}-e^{-u^{p}\kappa}\right) u^{-1-\alpha} \rd u &=&p^{-1} \int_0^\infty\left(e^{-v}-e^{-v\kappa}\right) v^{-\alpha/p-1}\rd v \\
&=& p^{-1}\Gamma(-\alpha/p)\left(1-\kappa^{\alpha/p}\right)\\
&=& \alpha^{-1}\Gamma(1-\alpha/p)\left(\kappa^{\alpha/p}-1\right),
\end{eqnarray*}
where the second line follows by (3.38) in \cite{Grabchak:2016book}.
\end{proof}

\begin{proof}[Proof of Theorem \ref{thrm: trans char func}]
From Lemma 17.1 in \cite{Sato:1999} it follows that $Y$ is a Markov process with temporally homogenous transition function $P_t(y,\rd x)$ on $\mathbb R^d$ having a characteristic function given by
$$
\exp\left\{ie^{-\lambda t}\langle y,z\rangle + i \left(1-e^{-\lambda t}\right) \langle  c_\alpha,z\rangle +\lambda\int_0^t \int_{\mathbb R^d}\psi_\alpha(e^{-\lambda s}z,x)M(\rd x)\rd s\right\}.
$$ 
We begin with the case $\alpha\ne1$. In this case it is easily checked that, for $a\in\mathbb R$, we have $\psi_\alpha(az,x)=\psi_\alpha(z,ax)$. It follows that
\begin{eqnarray*}
&&\lambda \int_0^t \int_{\mathbb R^d}\psi_\alpha(e^{-\lambda s}z,x) M(\rd x)\rd s \\
&&\qquad = \lambda \int_{\mathbb S^{d-1}}\int_0^\infty   \int_0^t  \psi_\alpha(z,e^{-\lambda s}r\xi) \rd s \rd \rho_\xi(r) \sigma(\rd \xi) \\
&&\qquad =\int_{\mathbb S^{d-1}}\int_0^\infty   \int_{re^{-\lambda t}}^r  \psi_\alpha(z,u\xi)u^{-1} \rd u \rd \rho_\xi(r) \sigma(\rd \xi) \\
&&\qquad =\int_{\mathbb S^{d-1}}\int_0^\infty \psi_\alpha(z,u\xi)  \int_0^\infty 1_{[re^{-\lambda t}\le u \le r]} \rd \rho_\xi(r) u^{-1} \rd u  \sigma(\rd \xi) \\
&&\qquad =\int_{\mathbb S^{d-1}}\int_0^\infty \psi_\alpha(z,u\xi) \left( q(\xi,u) - e^{-\alpha\lambda t}q(\xi,ue^{\lambda t}) \right)  u^{-1-\alpha} \rd u  \sigma(\rd \xi)\\
&&\qquad =  \left( 1 - e^{-\alpha\lambda t}\right)  \int_{\mathbb S^{d-1}}\int_0^\infty \psi_\alpha(z,u\xi) q(\xi,u)   u^{-1-\alpha} \rd u  \sigma(\rd \xi) \\
&&\qquad\quad +e^{-\alpha\lambda t} \int_{\mathbb S^{d-1}}\int_0^\infty \psi_\alpha(z,u\xi) \left( q(\xi,u) - q(\xi,ue^{\lambda t}) \right)  u^{-1-\alpha} \rd u  \sigma(\rd \xi),
\end{eqnarray*}
where the third line follows by the substitution $u=e^{-\lambda s}r$ and the fifth line by the fact that
\begin{eqnarray*}
 \int_0^\infty 1_{[e^{-\lambda t}r\le u \le r]} \rd \rho_\xi(r) &=&  \int_0^\infty 1_{[u\le r\le ue^{\lambda t}]} \rd \rho_\xi(r)= \rho_\xi(ue^{\lambda t}) - \rho_\xi(u) \\
&=&  \left( q(\xi,u) - e^{-\alpha\lambda t}q(\xi,ue^{\lambda t}) \right)  u^{-\alpha}.
\end{eqnarray*}

When $\alpha=1$ and $a\in\mathbb R$, we have
$$
\psi_1(az,x)=\psi_1(z,ax) + i a \langle z,x\rangle \left(1_{[|xa|\le 1]} - 1_{[|x|\le1]}\right).
$$
From here, proceeding as before, we get
\begin{eqnarray*}
&&\lambda \int_0^t \int_{\mathbb R^d}\psi_1(e^{-\lambda s}z,x) M(\rd x)\rd s \\
&&\qquad =  \left( 1 - e^{-\lambda t}\right)  \int_{\mathbb S^{d-1}}\int_0^\infty \psi_1(z,u\xi) q(\xi,u)   u^{-2} \rd u  \sigma(\rd \xi) \\
&&\qquad\quad +e^{-\lambda t} \int_{\mathbb S^{d-1}}\int_0^\infty \psi_1(z,u\xi) \left( q(\xi,u) - q(\xi,ue^{\lambda t}) \right)  u^{-2} \rd u  \sigma(\rd \xi)\\
&& \qquad\quad + i\lambda \int_0^t \int_{\mathbb R^d} \langle z,x\rangle \left(1_{|x|\le e^{\lambda s}} - 1_{|x|\le1}\right) M(\rd x) e^{-\lambda s}\rd s.
\end{eqnarray*}
Now, note that the third term equals
\begin{eqnarray*}
&&i \lambda\int_{\mathbb R^d} \langle z,x\rangle  \int_0^t\left(1_{1<|x|\le e^{\lambda s}} \right)e^{-\lambda s}\rd s M(\rd x)\\
&&\qquad=i \lambda\int_{1<|x|\le e^{\lambda t}} \langle z,x\rangle  \int_{\frac{\log|x|}{\lambda}}^t e^{-\lambda s}\rd s M(\rd x)\\
&&\qquad= i\int_{1<|x|\le e^{\lambda t}}  \langle z,x\rangle \left(|x|^{-1}-e^{-\lambda t}\right) M(\rd x).
\end{eqnarray*}

We now show the convergence to $\mathrm{TS}_{\alpha,d}(\sigma,q,b)$. It suffices to show that the characteristic function of the transition law approaches the characteristic function of $\mathrm{TS}_{\alpha,d}(\sigma,q,b)$ as $t\to\infty$. First, since $|q(\xi,u)|\le 1$ for each $\xi\in\mathbb S^{d-1}$ and $u>0$, we have
\begin{eqnarray*}
&&\lim_{t\to\infty} e^{-\alpha\lambda t}\left| \int_{\mathbb S^{d-1}}\int_0^\infty \psi_\alpha(z,u\xi) \left( q(\xi,u) - q(\xi,ue^{\lambda t}) \right)  u^{-1-\alpha} \rd u  \sigma(\rd \xi) \right|\\
&& \le  \lim_{t\to\infty} e^{-\alpha\lambda t}\int_{\mathbb S^{d-1}}\int_0^\infty\left|  \psi_\alpha(z,u\xi)\right|  u^{-1-\alpha} \rd u  \sigma(\rd \xi) =0,
\end{eqnarray*}
where the final equality follows from the fact that the integral is finite, which itself follows by standard facts about the function $\psi_\alpha$, see (26.4) in \cite{Billingsley:1995}. Now, note that $d_{\alpha}\to b$ as $t\to\infty$. This is immediate for $\alpha\ne1$ and  holds for $\alpha=1$ by dominated convergence. From here it follows that
\begin{eqnarray*}
\lim_{t\to\infty} \int_{\mathbb R^d}e^{i\langle x,y\rangle} P_t(y,\rd x)  = \exp\left\{i \langle  b,z\rangle + \int_{\mathbb R^d}\psi_\alpha(z,x) L_\alpha(\rd x)\right\},
\end{eqnarray*}
which is the characteristic function of $\mathrm{TS}_\alpha(\sigma,q,b)$ as required.
\end{proof}

\begin{proof}[Proof of Theorem \ref{thrm: main gen}]
We only prove \eqref{eq: trans RV} as the proof of \eqref{eq: trans RV stable} is similar. It suffices to check that the characteristic function of the random variable on the right side of \eqref{eq: trans RV} is given by $\exp\left\{C_t(y,z)\right\}$, where $C_t(y,z)$ is as in Theorem \ref{thrm: trans char func}. This follows immediately from the fact that the L\'evy measure of $\mathrm{TS}_{\alpha,d}(\sigma,q_0,0)$ is $(1-e^{-\alpha \lambda t})L_\alpha(\rd x)$ and $\sum_{j=1}^{N} X_j$ is a compound Poisson random variable, whose characteristic function is given in e.g.\ Proposition 3.4 of \cite{Cont:Tankov:2004}.
\end{proof}

\begin{proof}[Proof of Proposition \ref{prop:for AR}]
If $u\in(0, \epsilon e^{-\lambda t})$, then
$$
f_\xi(u)= \kappa_\xi \left( q(\xi,u) - q(\xi,ue^{\lambda t}) \right)  u^{-1-\alpha} \le \kappa_\xi M_\xi \left( e^{p(\xi)\lambda t} - 1 \right)  u^{p(\xi)-1-\alpha} .
$$
On the other hand, if $u\ge \epsilon e^{-\lambda t}$, then the fact that $q(\xi,u)\le1$ implies that
$$
f_\xi(u)= \kappa_\xi \left( q(\xi,u) - q(\xi,ue^{\lambda t}) \right)  u^{-1-\alpha} \le \kappa_\xi u^{-1-\alpha}
$$
and the result follows.
\end{proof}

\begin{proof}[Proof of Proposition \ref{prop: finite K}]
First, note that by change of variables
\begin{eqnarray*}
K &=& \int_{\mathbb S^{d-1}}\int_0^\infty  \left( q(\xi,u) - q(\xi,ue^{\lambda t}) \right)  u^{-1-\alpha} \rd u \sigma(\rd \xi) \\
&=& \int_0^\infty \int_{\mathbb S^{d-1}}\int_{(0,\infty)} \left( e^{-su^p} - e^{-su^pe^{p\lambda t}} \right)Q_\xi(\rd s) \sigma(\rd \xi) u^{-1-\alpha} \rd u  \\
&=&\int_0^\infty \int_{\mathbb R^d} \left( e^{-|x|u^p} - e^{-|x|u^pe^{p\lambda t}} \right)Q(\rd x) u^{-1-\alpha} \rd u  \\
&=&\int_0^\infty \left( e^{-v^p} - e^{-v^pe^{p\lambda t}} \right)v^{-1-\alpha} \rd v \int_{\mathbb R^d}|x|^{\alpha/p}Q(\rd x)   \\
&=& R(\mathbb R^d)  \int_0^\infty \left( e^{-v^p} - e^{-v^pe^{p\lambda t}} \right)v^{-1-\alpha} \rd v.
\end{eqnarray*}
Thus, if $R$ is an infinite measure, then $K=\infty$. Next, if $\alpha\ge p$, then Lemma \ref{lemma: exp fact} implies that 
\begin{eqnarray*}
 \int_0^\infty \left(e^{-u^p}-e^{-u^{p}e^{\lambda tp}}\right) u^{-1-\alpha} \rd u &\ge& \left(e^{\lambda tp}-1\right) \int_0^\infty e^{-u^{p}e^{\lambda tp}} u^{p-1-\alpha} \rd u\\
&=& \infty
\end{eqnarray*}
and we again have $K=\infty$. Henceforth, assume that $R$ is a finite measure and $p>\alpha$. In this case, Lemma \ref{lemma: integ diff} implies that
\begin{eqnarray*}
 \int_0^\infty \left(e^{-u^p}-e^{-u^{p}e^{\lambda tp}}\right) u^{-1-\alpha} \rd u &=& \alpha^{-1}\Gamma(1-\alpha/p)\left(e^{\lambda t\alpha}-1\right),
 \end{eqnarray*}
 which gives the form of $K$. Now assume that \eqref{eq: on R for F} holds. We again use  Lemma \ref{lemma: exp fact}, which gives, for any $u,v>0$,
\begin{eqnarray*}
\left|q(\xi,u) - q(\xi,v)\right| \le \int_{(0,\infty)} \left|e^{-su^p}-e^{-sv^p}\right| Q_\xi(\rd s) \le \left| u^p-v^p\right|  \int_{(0,\infty)} s Q_\xi(\rd s).
\end{eqnarray*}
From here the result follows by Lemma \ref{lemma: moments of Q}.
\end{proof}

\begin{proof}[Proof of Proposition \ref{prop: form pTaS}.]
Note that
\begin{eqnarray*}
\frac{1}{\kappa_\xi} &=& \int_0^\infty \left( q(\xi,u) - q(\xi,ue^{\lambda t}) \right) u^{-1-\alpha}\rd u\\
&=&\int_{(0,\infty)}\int_0^\infty \left( e^{-su^p} - e^{-su^pe^{p\lambda t}} \right) u^{-1-\alpha}\rd u Q_\xi(\rd s)\\
&=&\int_{(0,\infty)}\int_0^\infty \left( e^{-u^p} - e^{-u^pe^{p\lambda t}} \right) u^{-1-\alpha}\rd u s^{\alpha/p}Q_\xi(\rd s).
\end{eqnarray*}
From here the formula for $\kappa_\xi$ follows from Lemma \ref{lemma: integ diff}. The rest of the formulas are immediate.
\end{proof}

\begin{proof}[Proof of Proposition \ref{prop: alt bound}]
Lemma \ref{lemma: exp fact} implies that
\begin{eqnarray*}
u^{-1-\alpha}\left(q_\xi(u) - q_\xi(ue^{t\lambda}) \right) &=&  u^{-1-\alpha} \int_{(0,\infty)}\left( e^{-u^{p}r} - e^{-u^{p} e^{t\lambda p}r } \right) Q_\xi(\rd r)\\
&\le&  \left(e^{t\lambda p}-1\right) u^{p-1-\alpha} \int_{[\zeta,\infty)}e^{-u^{p}r} rQ_\xi(\rd r)\\
&\le& \left(e^{t\lambda p}-1\right) u^{p-1-\alpha} e^{-u^{p}\zeta}  \int_{(0,\infty)}rQ_\xi(\rd r),
\end{eqnarray*}
which gives the result.
\end{proof}

\begin{proof}[Proof of Propositions \ref{prop: for example 1} and \ref{prop: for example}]
The results in Proposition \ref{prop: for example 1} follow immediately from \eqref{eq: back to Q}, \eqref{eq:back to sig}, and \eqref{eq: Q}. We now turn to Propositions \ref{prop: for example}. Let $\phi(x) = (\alpha+\ell+1)^{-1}\left(\frac{x}{x+1}\right)^{\alpha+\ell+1}$ for $x>0$ and note that $\phi'(x) = x^{\alpha+\ell} (x+1)^{-\alpha-\ell-2}$. It follows that 
$$
B\int_{0}^\infty x^\alpha Q_\xi(\rd x) =\int_{0}^\infty (1+x)^{-2-\alpha-\ell} x^{\ell+\alpha}\rd x = \frac{1}{\alpha+\ell+1}.
$$
Using this fact we can easily put the formulas into the required form.
\end{proof}

\end{document}